\documentclass{jms_my}
\paperTitle{Two forms of the integral representations of the Mittag-Leffler function}
\articleColonName{Two forms of the Mittag-Leffler function}
\authorsShort{V.\,V.~Saenko}
\authorsFull{V.\,V. Saenko\first}
\addAuthorInfo{Ulyanovsk State University, S.P. Kapitsa Research Institute of Technology,   L. Tolstoy St. 42, Ulyanovsk,  Russia, 432017 e-mail: \url{saenkovv@gmail.com}}

\paperAbstract{The integral representation of the two-parameter Mittag-Leffler function $E_{\rho,\mu}(z)$ is considered in the paper that expresses its  value in terms of the contour integral. For this integral representation, the transition is made from integration over a complex variable to integration over real variables. It is shown that as a result of such a transition, the integral representation of the function  $E_{\rho,\mu}(z)$ has two forms: the representation ``A'' and ``B''. Each of these representations has its advantages and drawbacks. In the paper, the corresponding theorems are formulated and proved, and the advantages and disadvantages of each of the obtained representations are discussed.}

\begin{document}
\maketitle

\section{Introduction}

The Mittag-Leffler function is an entire function defined by a power series
\begin{equation*}
  E_{\rho}(z)=\sum_{k=0}^{\infty}\frac{z^k}{\Gamma(1+k/\rho)},\quad \rho>0,\quad z\in\C,
\end{equation*}
where $\Gamma(x)$ is the Gamma-function. This function was introduced by Mittag-Leffler in a number of works published from 1902 to 1905 in connection with  his development of a method of summing divergent series. For~more detailed information on the content of these works and on the history of the introduction of the Mittag-Leffler function, we refer the reader to the book~\cite{Gorenflo2014} (see chapter 2 in~\cite{Gorenflo2014}). The~function itself $E_{\rho}(z)$ was introduced in the work~\cite{Mittag-Leffler1903}. In~the paper~\cite{Mittag-Leffler1905}, the integral representation for this function was obtained that expresses its value in terms of the contour~integral.

In this paper the two parameter Mittag-Leffler function
\begin{equation}\label{eq:MLF_gen}
  E_{\rho,\mu}(z)=\sum_{k=0}^{\infty}\frac{z^k}{\Gamma(\mu+k/\rho)},\quad \rho>0,\quad \mu\in\C,\quad z\in\C
\end{equation}
is studied. This function was first introduced by A. Wiman  in 1905~\cite{Wiman1905,Wiman1905a}. Later in  1953 this function was rediscovered in the works of  Humbert  and Agarval~\cite{zbMATH03082752,zbMATH03078845,zbMATH03081895}. A~new function was introduced by replacing the additive unit in the Gamma function argument in $E_{\rho}(z)$  for an arbitrary complex parameter $\mu$. At~the same time, irrespective of  Humbert  and Agarval, the~function  (\ref{eq:MLF_gen}) was studied by M.M. Djrbashian  in the papers~\cite{Dzhrbashyan1954_eng,Dzhrbashian1954b_eng} (see also~\cite{Dzhrbashyan1966_eng} Chapter 3, \S 2, 4).  As~we can see, the~two parameter Mittag-Leffler function $E_{\rho,\mu}(z)$ is connected with the classic Mittag-Leffler function  $E_{\rho}(z)$ by a simple relation $E_{\rho,1}(z)=E_{\rho}(z)$. For~more detailed information on the properties of the function  $E_{\rho,\mu} (z)$ we refer the reader to the book~\cite{Gorenflo2014}, as~well as to other review works~\cite{Popov2013,Rogosin2015,Gorenflo2019}. In~this paper, integral representations of the function $E_{\rho,\mu}(z)$ will be obtained and~studied.

The integral representation of the Mittag-Leffler function is important from the point of view of its practical use, as~well as for studying the asymptotic properties and zeros of this function. The~integral representation expressed through the contour integral is used for these purposes. Several such representations are known for the Mittag-Leffler function. One of the earliest integral representations of the function $E_{\rho,\mu}(z)$ was given in the book~\cite{Bateman_V3_1955} (see \S18.1, formula~(20)).  Further development of the issue of the integral representation of the Mittag-Leffler function and the study of its asymptotic properties was carried out in the works of M.M. Djrbashian. In~the work~\cite{Dzhrbashyan1954_eng} the integral representation was obtained that expressed  the Mittag-Leffler function through the contour integral. Later it was included in his monograph~\cite{Dzhrbashyan1966_eng} (see chapter 3, \S2, Lemma~3.2.1). Using this integral representation, asymptotic formulas and the distribution of the zeros of the Mittag-Leffler function were obtained. Further, the~results of the work~\cite{Dzhrbashyan1954_eng,Dzhrbashyan1966_eng} were used in the books of~\cite{Gorenflo2014, Podlubny1999}, as~well as in the works of~\cite{Gorenflo2002b,Seybold2005,Hilfer2006,Seybold2009,Parovik2012_eng} to develop the methods and algorithms of calculating the Mittag-Leffler function.  However, despite a wide use of the integral representation for  the Mittag-Leffler function that was obtained in the works of~\cite{Dzhrbashyan1954_eng,Dzhrbashyan1966_eng}, it turned out that there was a mistake in that representation. This fact was pointed out in the work of~\cite{Saenko2020}.  In~this regard, there is an  issue of obtaining the correct integral representation for the Mittag-Leffler~function.

One of the possible solutions to this issue was given in the works of~\cite{Saenko2020,saenko2019}. In~these works, the~representation of the Mittag-Leffler function was obtained that expresses its value through the contour integral. As~it was noted earlier, this representation is used to study the asymptotic properties of the Mittag-Leffler function. However, for~practical use and for calculating the value of the function, it is convenient to have integral representations expressing the function in terms of the integrals of real variables. This paper is devoted to obtaining such integral representations for the Mittag-Leffler function.  The~starting point of the solution to this problem is the integral representation of the function $E_{\rho,\mu}(z)$ obtained in the work~\cite{Saenko2020}. Running a little ahead, we will say that the transition from integration over a complex variable to integration over real variables leads to the appearance of two forms of the integral representation of the Mittag-Leffler function. The~first form will be abbreviated as the representation ``A'', %can the quotes on A and B be removed? please confirm and update throughout.
 the~second as the representation  ``B''. The~representation ``A'' is a direct  consequence of the integral representation obtained in the work~\cite{Saenko2020}. It is obtained as a result of the transition from the contour integral to integration over real variables. To~obtain the representation  ``B'' in addition to performing such a transition, it is necessary to carry out a terminal transition  $\varepsilon\to0$. This leads to the fact that the integral representation ``B'' is valid only for parameter values  $\mu$ satisfying the condition $\Re\mu<1+1/\rho$. As~a result, both the representation ``A'' and the representation ``B'' have its advantages and drawbacks which will be discussed in detail in the~paper.

It should be pointed out that  in this paper the letters $\rho,\mu$ are used to denote the parameters of the Mittag-Leffler function (\ref{eq:MLF_gen}).   These notations were introduced by M.M. Djrbashian in his works~\cite{Dzhrbashyan1954_eng,Dzhrbashian1954b_eng,Dzhrbashyan1966_eng}. Parameters $\rho$ and $\mu$ of the function (\ref{eq:MLF_gen}) are connected with the commonly used notations  $\alpha$ and $\beta$ of the parameters of the Mittag-Leffler function by simple relations $\alpha=1/\rho$, $\beta=\mu$.

\section{Integral representation  ``A''}

The purpose of this paper  is to obtain integral representations for the function $E_{\rho,\mu}(z)$ expressing this function in terms of integrals over real variables. It is convenient to use the representations of such a kind in practical problems as well as for calculating the values of the Mittag-Leffler function.  As~an example of the use of such integral representations, we can mention the work~\cite{Saenko2020c}. In~that paper, the~inverse Fourier transform of the characteristic function of the fractionally stable law was performed which is expressed in terms of the Mittag-Leffler function. To~perform the inverse Fourier transform, we used the integral representation of the Mittag-Leffler function which in the current paper is formulated in corollary~\ref{coroll:MLF_int1_deltaRho} item 2 (see (\ref{eq:MLF_int1_piRho})). Using this integral representation in the paper~\cite{Saenko2020c} expressions of density and distribution function of a fractionally stable law were~obtained.  Articles \cite{UCHAIKIN2008,Uchaikin2009a} are another example of usage of the Mittag-Leffler function. In these articles are shown that solutions of a master equation for the fractional Poisson process \cite{UCHAIKIN2008} and a fractional relaxation equation for dielectrics \cite{Uchaikin2009a} are expressed thorough the Mittag-Leffler function. A Monte Carlo method was used for calculation the obtained solutions. However,  the use of the integral representation of the Mittag-Leffler function for calculation of the solutions would significantly increase the accuracy of the results.

The starting point of this paper is the integral representation of the function $E_{\rho,\mu}(z)$, obtained in the work~\cite{Saenko2020}. The~following theorem was formulated in this~work

\begin{theorem}\label{lemm:MLF_int}
For any real $\rho, \delta_{1\rho}, \delta_{2\rho}, \epsilon$  satisfying the conditions $\rho>1/2$, $\frac{\pi}{2\rho}<\delta_{1\rho}\leqslant\min(\pi,\pi/\rho)$, $\frac{\pi}{2\rho}<\delta_{2\rho}\leqslant\min(\pi,\pi/\rho)$, $\epsilon>0$, any $\mu\in\C$ and any $z\in\C$  satisfying the condition
\begin{equation}\label{eq:z_cond_lemm}
  \frac{\pi}{2\rho}-\delta_{2\rho}+\pi<\arg z<-\frac{\pi}{2\rho}+\delta_{1\rho}+\pi
\end{equation}
the Mittage-Leffler function can be represented in the from
\begin{equation}\label{eq:MLF_int}
  E_{\rho,\mu}(z)=\frac{\rho}{2\pi i} \int_{\gamma_\zeta}\frac{\exp\left\{(z\zeta)^{\rho}\right\}(z\zeta)^{\rho(1-\mu)}}{\zeta-1}d\zeta.
\end{equation}
where the contour of integration $\gamma_\zeta$ has the form (see~Figure~\ref{fig:loop_gammaZeta})
\begin{equation}\label{eq:loop_gammaZeta}
  \gamma_\zeta=\left\{\begin{array}{ll}
                       S_1=&\{\zeta: \arg\zeta=-\delta_{1\rho}-\pi,\quad |\zeta|\geqslant 1+\epsilon\},\\
                       C_\epsilon=&\{\zeta: -\delta_{1\rho}-\pi\leqslant\arg\zeta\leqslant\delta_{2\rho}- \pi,\quad |\zeta|=1+\epsilon\},\\
                       S_2=&\{\zeta: \arg\zeta=\delta_{2\rho}-\pi,\quad|\zeta|\geqslant1+\epsilon\}.
                     \end{array}\right.
\end{equation}
\end{theorem}

\begin{figure}[H]
  \centering
  \includegraphics[width=0.4\textwidth]{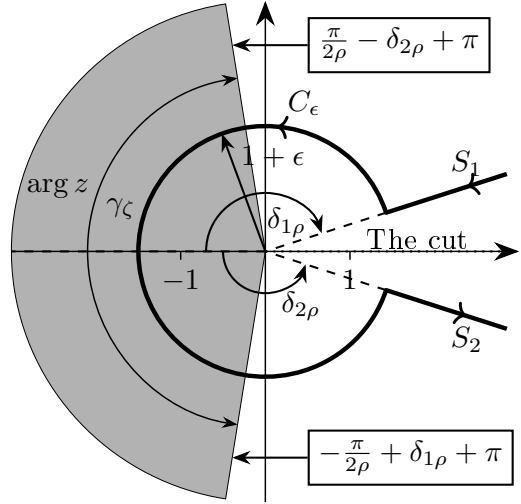}
  \caption{The contour of integration $\gamma_\zeta$. The~region that corresponds to the condition (\ref{eq:z_cond_lemm}) is shaded in~grey.}\label{fig:loop_gammaZeta}
\end{figure}

The proof of this theorem can be found in the work~\cite{Saenko2020}.

The form of the contour of integration $\gamma_\zeta$ on the complex plane  $\zeta$ is given in Figure~\ref{fig:loop_gammaZeta}. The~cut of the complex plane goes along the positive part of a real axis.  The~contour of integration consists of the half-line $S_1$, the~arc of the circle  $C_\epsilon$ radius $1+\epsilon$ and the half-line $S_2$.   The~contour $\gamma_\zeta$ is traversed in a positive direction. The~parameters $\delta_{1\rho}$ and $\delta_{2\rho}$ have the meaning of inclination angles of the half-lines $S_1$ and $S_2$ in relation to the contour axis $\gamma_\zeta$. In~Figure~\ref{fig:loop_gammaZeta} the contour axis $\gamma_\zeta$ coincides with a real axis. The~values of the angles $\delta_{1\rho}$ and $\delta_{2\rho}$  are measured from the negative part of the real axis. The~reference directions of these angles in Figure~\ref{fig:loop_gammaZeta} are shown by~arrows.

We will obtain the integral representations that of interest to us by going from integration over the complex variable $\zeta$ to integration over variables $r$ and $\varphi$ interconnected by the relation $\zeta=re^{i\varphi}$. As~a result, the~following theorem is true for the function $E_{\rho,\mu}(z)$.

\begin{theorem}[The representation ``A'']\label{lemm:MLF_int0}
For any real $\rho>1/2$, $\epsilon>0$  and $\delta_{1\rho}$, $\delta_{2\rho}$ satisfying the conditions
\begin{equation}\label{eq:deltaCond_lemm_int0}
 \frac{\pi}{2\rho}<\delta_{1\rho}\leqslant\min\left(\pi,\frac{\pi}{\rho}\right),\quad \frac{\pi}{2\rho}<\delta_{2\rho}\leqslant\min\left(\pi,\frac{\pi}{\rho}\right),
\end{equation}
for any complex $\mu$ and any complex $z$ satisfying the condition
\begin{equation}\label{eq:lem_CommonRepr_argZcond}
  \frac{\pi}{2\rho}-\delta_{2\rho}+\pi<\arg z<-\frac{\pi}{2\rho}+\delta_{1\rho}+\pi,
\end{equation}
the Mittag-Leffler function can be represented in the form
\begin{equation}\label{eq:MLF_int0}
  E_{\rho,\mu}(z)=\int_{1+\epsilon}^{\infty}K_{\rho,\mu}(r,-\delta_{1\rho},\delta_{2\rho},z)dr+
  \int_{-\delta_{1\rho}-\pi}^{\delta_{2\rho}-\pi}P_{\rho,\mu}(1+\epsilon,\varphi,z)d\varphi.
\end{equation}

Here
\begin{equation}\label{eq:K_lemm_int0}
  K_{\rho,\mu}(r,\varphi_1,\varphi_2,z) =\frac{\rho}{2\pi i} \frac{(zr)^{\rho(1-\mu)}e^{-i\rho\pi(1-\mu)}[A_{\rho,\mu}(r,\varphi_2,\varphi_1,z)-A_{\rho,\mu}(r,\varphi_1,\varphi_2,z)]}
  {(r^2+2r\cos\varphi_1+1) (r^2+2r\cos\varphi_2+1)},
\end{equation}
where
\begin{equation}\label{eq:A_lemm_int0}
  A_{\rho,\mu}(r,\omega_1,\omega_2,z) \\
  = \exp\left\{(zr)^\rho e^{-i\rho\pi}\cos(\rho\omega_1)\right\}(r^2+2r\cos\omega_2+1)e^{i\eta(r,\omega_1,z)}\left[r+e^{i\omega_1}\right],
\end{equation}
\begin{equation}\label{eq:eta_lemm_int0}
  \eta(r,\varphi,z)=(zr)^{\rho}e^{-i\rho\pi}\sin(\rho\varphi)+\rho(1-\mu)\varphi,
\end{equation}
and
\begin{equation}\label{eq:lemm_int0_P}
  P_{\rho,\mu}(r,\varphi,z)=\frac{\rho}{2\pi}\frac{\exp\left\{(z r)^\rho\cos(\rho\varphi)\right\} (z r)^{\rho(1-\mu)}e^{i\chi(r,\varphi,z)}r\left[r-e^{i\varphi}\right]} {r^2-2r\cos\varphi+1},
\end{equation}
where
\begin{equation}\label{eq:chi_lemm_int0}
  \chi(r,\varphi,z)=(zr)^\rho\sin(\rho\varphi)+\rho(1-\mu)\varphi.
\end{equation}
\end{theorem}

\begin{proof}
For convenience we will introduce the notation
\begin{equation}\label{eq:phiFun}
  \phi_{\rho,\mu}(\zeta,z)=\frac{\rho}{2\pi i}\exp\left\{(\zeta z)^\rho\right\}(\zeta z)^{\rho(1-\mu)}.
\end{equation}
As a result, the~representation (\ref{eq:MLF_int}) will take the form
\begin{equation}\label{eq:MLF_tmp6}
  E_{\rho,\mu}(z)=  \int_{\gamma_\zeta}\frac{\phi_{\rho,\mu}(\zeta,z)}{\zeta-1}d\zeta,
\end{equation}
where the contour of integration $\gamma_\zeta$ is defined by (\ref{eq:loop_gammaZeta}).

Substituting in (\ref{eq:MLF_tmp6}) the variable of integration $\zeta=re^{i\varphi}$ and  directly calculating this integral we~obtain
\begin{multline}\label{eq:MLF_commonRepr}
  E_{\rho,\mu}(z)=\int_{S_1}\frac{\phi_{\rho,\mu}(\zeta,z)}{\zeta-1}d\zeta+ \int_{C_\epsilon}\frac{\phi_{\rho,\mu}(\zeta,z)}{\zeta-1}d\zeta+ \int_{S_2}\frac{\phi_{\rho,\mu}(\zeta,z)}{\zeta-1}d\zeta
  =\int_{\infty}^{1+\epsilon}\left.\frac{\phi_{\rho,\mu}(re^{i\varphi},z)}{re^{i\varphi}-1}e^{i\varphi}dr\right|_{\varphi=-\delta_{1\rho}-\pi}\\
  +i\hspace{-2mm}\int\limits_{-\delta_{1\rho}-\pi}^{\delta_{2\rho}-\pi}\hspace{-2mm}\left.\frac{\phi_{\rho,\mu}(re^{i\varphi},z)}{re^{i\varphi}-1}re^{i\varphi}d\varphi\right|_{r=1+\epsilon}
  +\hspace{-2mm}\int\limits_{1+\epsilon}^{\infty}\left.\frac{\phi_{\rho,\mu}(re^{i\varphi},z)}{re^{i\varphi}-1}e^{i\varphi}dr\right|_{\varphi=\delta_{2\rho}-\pi}
  =\int\limits_{\infty}^{1+\epsilon}\frac{\phi_{\rho,\mu}(re^{i(-\delta_{1\rho}-\pi)},z)}{re^{i(-\delta_{1\rho}-\pi)}-1}e^{i(-\delta_{1\rho}-\pi)}dr\\
  +i\hspace{-2mm}\int\limits_{-\delta_{1\rho}-\pi}^{\delta_{2\rho}-\pi}\hspace{-2mm}\frac{\phi_{\rho,\mu}((1+\epsilon)e^{i\varphi},z)}{(1+\epsilon)e^{i\varphi}-1}(1+\epsilon)e^{i\varphi}d\varphi
  +\int\limits_{1+\epsilon}^{\infty}\frac{\phi_{\rho,\mu}(re^{i(\delta_{2\rho}-\pi)},z)}{re^{i(\delta_{2\rho}-\pi)}-1}e^{i(\delta_{2\rho}-\pi)}dr
  =I_{S_1}+I_{C_\epsilon}+I_{S_2}.
\end{multline}

We consider the integral $I_{S_1}$. Getting rid of the complexity in the denominator we have
\begin{multline*}
  I_{S_1}=\int_{\infty}^{1+\epsilon}\frac{\phi_{\rho,\mu}\left(re^{i(-\delta_{1\rho}-\pi)},z\right)e^{i(-\delta_{1\rho}-\pi)}} {re^{i(-\delta_{1\rho}-\pi)}-1}dr = \int_{\infty}^{1+\epsilon} \frac{\phi_{\rho,\mu}\left(re^{i(-\delta_{1\rho}-\pi)},z\right)e^{i(-\delta_{1\rho}-\pi)} \left(re^{i(\delta_{1\rho}+\pi)}-1\right)} {\left(re^{i(-\delta_{1\rho}-\pi)}-1\right)\left(re^{i(\delta_{1\rho}+\pi)}-1\right)}dr\\
  = \int_{\infty}^{1+\epsilon}\frac{\phi_{\rho,\mu}\left(re^{i(-\delta_{1\rho}-\pi)},z\right) \left(r+e^{-i\delta_{1\rho}}\right) }{r^2+2r\cos\delta_{1\rho}+1} dr\\
  =\frac{\rho}{2\pi i}\int_{\infty}^{1+\epsilon} \frac{\exp\left\{\left(rze^{-i(\delta_{1\rho}+\pi)}\right)^\rho\right\} \left(zre^{-i(\delta_{1\rho}+\pi)}\right)^{\rho(1-\mu)} \left(r+e^{-i\delta_{1\rho}}\right)}{r^2+2r\cos\delta_{1\rho}+1}dr\\
  =\frac{\rho}{2\pi i}\int_{\infty}^{1+\epsilon}\frac{\exp\left\{(zr)^\rho e^{-i\rho\pi}(\cos(\rho\delta_{1\rho})-i\sin(\rho\delta_{1\rho}))\right\}}{r^2+2r\cos\delta_{1\rho}+1}
  (zr)^{\rho(1-\mu)}e^{-i\rho\pi(1-\mu)}e^{-i\rho(1-\mu)\delta_{1\rho}}(r+e^{-i\delta_{1\rho}})dr\\
  =\frac{\rho}{2\pi i}\int_{\infty}^{1+\epsilon}\frac{\exp\left\{(zr)^\rho e^{-i\rho\pi}\cos(\rho\delta_{1\rho})\right\}}{r^2+2r\cos\delta_{1\rho}+1}
 (zr)^{\rho(1-\mu)}e^{-i\rho\pi(1-\mu)}e^{i\eta(r,-\delta_{1\rho},z)} \left(r+e^{-i\delta_{1\rho}}\right)dr,
\end{multline*}
where the notation $ \eta(r,\varphi,z)=(zr)^\rho e^{-i\rho\pi}\sin(\rho\varphi)+\rho(1-\mu)\varphi$ was~introduced.

Similarly, for~the integral $I_{S_2}$ we get
\begin{multline*}
  I_{S_2}=
  \int_{1+\epsilon}^{\infty}\frac{\phi_{\rho,\mu}\left(re^{i(\delta_{2\rho}-\pi)},z\right)e^{i(\delta_{2\rho}-\pi)}}{re^{i(\delta_{2\rho}-\pi)}-1}dr
  = \int_{1+\epsilon}^{\infty}\frac{\phi_{\rho,\mu}\left(re^{i(\delta_{2\rho}-\pi)},z\right) \left(r+e^{i\delta_{2\rho}}\right) }{r^2+2r\cos\delta_{2\rho}+1} dr\\
  = \frac{\rho}{2\pi i}\int_{1+\epsilon}^{\infty}\frac{\exp\left\{(zr)^\rho e^{-i\rho\pi}\cos(\rho\delta_{2\rho})\right\}(zr)^{\rho(1-\mu)}e^{-i\rho\pi(1-\mu)}e^{i\eta(r,\delta_{2\rho},z)} \left(r+e^{i\delta_{2\rho}}\right)}{r^2+2r\cos\delta_{2\rho}+1}dr.
\end{multline*}
Summing now the integrals $I_{S_1}$ and $I_{S_2}$ we obtain
\begin{multline}\label{eq:Is1+Is2}
  I_{S_1}+I_{S_2}= \frac{\rho}{2\pi i}\int_{\infty}^{1+\epsilon}\frac{\exp\left\{(zr)^\rho e^{-i\rho\pi}\cos(\rho\delta_{1\rho})\right\}}{r^2+2r\cos\delta_{1\rho}+1}
   (zr)^{\rho(1-\mu)}e^{-i\rho\pi(1-\mu)}e^{i\eta(r,-\delta_{1\rho},z)} \left(r+e^{-i\delta_{1\rho}}\right)dr\\
  + \frac{\rho}{2\pi i}\int_{1+\epsilon}^{\infty}\frac{\exp\left\{(zr)^\rho e^{-i\rho\pi}\cos(\rho\delta_{2\rho})\right\}(zr)^{\rho(1-\mu)}e^{-i\rho\pi(1-\mu)}e^{i\eta(r,\delta_{2\rho},z)} \left(r+e^{i\delta_{2\rho}}\right)}{r^2+2r\cos\delta_{2\rho}+1}dr\\
  =\int_{1+\epsilon}^{\infty}K_{\rho,\mu}(r,-\delta_{1\rho},\delta_{2\rho},z)dr,
\end{multline}
where
\begin{align*}
  K_{\rho,\mu}(r,\varphi_1,\varphi_2,z)&=\frac{\rho}{2\pi i} \frac{(zr)^{\rho(1-\mu)}e^{-i\rho\pi(1-\mu)}[A_{\rho,\mu}(r,\varphi_2,\varphi_1,z)-A_{\rho,\mu}(r,\varphi_1,\varphi_2,z)]}
  {(r^2+2r\cos\varphi_1+1) (r^2+2r\cos\varphi_2+1)},\\
  A_{\rho,\mu}(r,\omega_1,\omega_2,z) &= \exp\left\{(zr)^\rho e^{-i\rho\pi}\cos(\rho\omega_1)\right\}(r^2+2r\cos\omega_2+1)e^{i\eta(r,\omega_1,z)}\left[r+e^{i\omega_1}\right].
\end{align*}

We consider now the integral $I_{C_\epsilon}$.   At~the beginning, we will get rid of the complexity in the denominator. To~do this, we will multiply and divide the integrand by the complex conjugate of the denominator and open the brackets in the denominator. Then, in~the resulting expression, we substitute the definition of the function  $\phi_{\rho,\mu}(\zeta,z)$ (see (\ref{eq:phiFun})) and in the indices of the exponents we use the Euler formula  $e^{i\varphi}=\cos\varphi+i\sin\varphi$. As~a result, we get
\begin{multline}\label{eq:Ic_eps}
  I_{C_\epsilon}= i\hspace{-3mm}\int\limits_{-\delta_{1\rho}-\pi}^{\delta_{2\rho}-\pi}\frac{\phi_{\rho,\mu}((1+\epsilon)e^{i\varphi},z)} {(1+\epsilon)e^{i\varphi}-1}(1+\epsilon)e^{i\varphi}d\varphi
  = i\hspace{-3mm}\int\limits_{-\delta_{1\rho}-\pi}^{\delta_{2\rho}-\pi}\frac{\phi_{\rho,\mu}((1+\epsilon)e^{i\varphi},z)(1+\epsilon)e^{i\varphi} \left((1+\epsilon)e^{-i\varphi}-1\right) } {\left((1+\epsilon)e^{i\varphi}-1\right) \left((1+\epsilon)e^{-i\varphi}-1\right)}d\varphi\\
  = i\int_{-\delta_{1\rho}-\pi}^{\delta_{2\rho}-\pi}\frac{\phi_{\rho,\mu}((1+\epsilon)e^{i\varphi},z)(1+\epsilon) \left((1+\epsilon)-e^{i\varphi}\right) } {(1+\epsilon)^2-2(1+\epsilon)\cos\varphi+1}d\varphi\\
  =\frac{\rho}{2\pi}\int_{-\delta_{1\rho}-\pi}^{\delta_{2\rho}-\pi} \frac{\exp\left\{\left(z(1+\epsilon)e^{i\varphi}\right)^\rho\right\} \left(z(1+\epsilon)e^{i\varphi}\right)^{\rho(1-\mu)} (1+\epsilon)\left((1+\epsilon)-e^{i\varphi}\right)}{(1+\epsilon)^2-2(1+\epsilon)\cos\varphi+1}d\varphi\\
  =\frac{\rho}{2\pi}\int\limits_{-\delta_{1\rho}-\pi}^{\delta_{2\rho}-\pi} \frac{\exp\left\{((1+\epsilon)z)^\rho(\cos(\rho\varphi) + i\sin(\rho\varphi))\right\}}{(1+\epsilon)^2-2(1+\epsilon)\cos\varphi+1}
  (z(1+\epsilon))^{\rho(1-\mu)} e^{i\rho\varphi(1-\mu)} (1+\epsilon)\left((1+\epsilon)-e^{i\varphi}\right)d\varphi\\
  =\frac{\rho}{2\pi}\int\limits_{-\delta_{1\rho}-\pi}^{\delta_{2\rho}-\pi}
  \frac{\exp\left\{((1+\epsilon)z)^\rho\cos(\rho\varphi)\right\}}{(1+\epsilon)^2-2(1+\epsilon)\cos\varphi+1}
   (z(1+\epsilon))^{\rho(1-\mu)} e^{i\chi((1+\epsilon),\varphi,z)}(1+\epsilon)\left((1+\epsilon)-e^{i\varphi}\right) d\varphi\\
  =\int_{-\delta_{1\rho}-\pi}^{\delta_{2\rho}-\pi} P_{\rho,\mu}(1+\epsilon,\varphi,z) d\varphi,
\end{multline}
where
\begin{align*}
  P_{\rho,\mu}(r,\varphi,z)&=\frac{\rho}{2\pi}\frac{\exp\left\{(zr)^\rho\cos(\rho\varphi)\right\} (zr)^{\rho(1-\mu)}e^{i\chi(r,\varphi,z)}r\left[r-e^{i\varphi}\right]} {r^2-2r\cos\varphi+1},\\
  \chi(r,\varphi,z)&=(zr)^\rho\sin(\rho\varphi)+\rho(1-\mu)\varphi.
\end{align*}

Using (\ref{eq:Is1+Is2}) and (\ref{eq:Ic_eps}) in (\ref{eq:MLF_commonRepr}) we  obtain the representation (\ref{eq:MLF_int0}). It is important to pay attention that in the process of proving no additional limitations on the values of parameters $\rho$, $\mu$, $\delta_{1\rho}$, $\delta_{2\rho}$ and the argument $z$ were imposed and it means that the  ranges of admissible values pass   from Theorem~\ref{lemm:MLF_int} without change. Thus, the~representation (\ref{eq:MLF_int0}) is valid for any real $\rho>1/2$, any real $\delta_{1\rho}$, $\delta_{2\rho}$ satisfying the conditions (\ref{eq:deltaCond_lemm_int0}), any complex $\mu$ and any complex $z$ satisfying the condition (\ref{eq:lem_CommonRepr_argZcond}).
\end{proof}

The proved theorem formulates an integral representation for the Mittag-Leffler function that expresses this function in terms of the sum of improper and definite integrals. To~be definite, we will call this integral representation of the Mittag-Leffler function the representation ``A''.  As~we can see, the~representation ``A'' is a direct consequence of the representation (\ref{eq:MLF_int}). It is obtained by passing from the contour integral to integrals over the real variable. Moreover, the~improper integral in (\ref{eq:MLF_int0}) corresponds to the sum of integrals along the half-lines $S_1$ and $S_2$ of the contour $\gamma_\zeta$ and the definite integral corresponds to the integral along the arc of a circle $C_\epsilon$. It should be noted that this integral is taken along the arc of a circle of radius $1+\epsilon$, where $\epsilon>0$.  The~representation (\ref{eq:MLF_int0}) is valid for arbitrary values $\rho>1/2$, any $\mu$ and any $\delta_{1\rho}$ and $\delta_{2\rho}$ that satisfy the condition (\ref{eq:deltaCond_lemm_int0}).

However, in~general case, for~arbitrary values of parameters $\delta_{1\rho}$ and $\delta_{2\rho}$, satisfying the condition~(\ref{eq:deltaCond_lemm_int0}),  the~kernel function $K_{\rho,\mu}(r,-\delta_{1\rho},\delta_{2\rho},z)$ turns to be lengthy. The~representation (\ref{eq:MLF_int0}) takes the more compact form in case when the half-lines $S_1$ and $S_2$ of the contour of integration $\gamma_\zeta$ run symmetrically relative to the real axis i.e.,~when $\delta_{1\rho}=\delta_{2\rho}=\delta_\rho$. We will formulate the obtained result in the form of a~corollary.

\begin{corollary}\label{coroll:MLF_int0_deltaRho}
For any real $\epsilon>0$, any complex $\mu$ the following integral representations of the Mittag-Leffler function are~true:
\begin{enumerate}
\item at any real $\rho>1/2$, any real $\delta_\rho$  satisfying the condition $\frac{\pi}{2\rho}<\delta_\rho\leqslant\min(\pi,\frac{\pi}{\rho})$ and any complex $z$ satisfying the condition
$\frac{\pi}{2\rho}-\delta_\rho+\pi<\arg z<-\frac{\pi}{2\rho}+\delta_\rho+\pi$
\begin{equation}\label{eq:MLF_int0_deltaRho}
  E_{\rho,\mu}(z)=\int_{1+\epsilon}^{\infty}K_{\rho,\mu}(r,\delta_\rho,z)dr+ \int_{-\delta_\rho-\pi}^{\delta_\rho-\pi}P_{\rho,\mu}(1+\epsilon,\varphi,z)d\varphi,
\end{equation}
where
\begin{equation}\label{eq:K_coroll_int0}
  K_{\rho,\mu}(r,\varphi,z)=\frac{\rho}{\pi}(zre^{-i\pi})^{\rho(1-\mu)}\exp\left\{(zre^{-i\pi})^\rho \cos(\rho\varphi)\right\}
  \frac{r\sin(\eta(r,\varphi,z))+ \sin(\eta(r,\varphi,z)+\varphi)}{r^2+2r\cos\varphi+1},
\end{equation}
$\eta(r,\varphi,z)$ is defined by (\ref{eq:eta_lemm_int0}) and $P_{\rho,\mu}(r,\varphi,z)$ has the form (\ref{eq:lemm_int0_P}).

\item at any real $\rho\geqslant1$ at $\delta_\rho=\pi/\rho$ and any complex $z$ satisfying the condition
$-\frac{\pi}{2\rho}+\pi<\arg z<\frac{\pi}{2\rho}+\pi$
\begin{equation}\label{eq:MLF_int0_piRho}
  E_{\rho,\mu}(z)=\int_{1+\epsilon}^{\infty}K_{\rho,\mu}(r,z)dr+ \int_{-\tfrac{\pi}{\rho}-\pi}^{\tfrac{\pi}{\rho}-\pi}P_{\rho,\mu}(1+\epsilon,\varphi,z)d\varphi,
\end{equation}
where
\begin{equation}\label{eq:K_piRho_coroll_int0}
  K_{\rho,\mu}(r,z)=\frac{\rho}{\pi}(zre^{-i\pi})^{\rho(1-\mu)} \exp\left\{-(zre^{-i\pi})^\rho \right\}
  \frac{r\sin(\pi(1-\mu))+\sin(\pi(1-\mu)+\pi/\rho)}{r^2+2r\cos(\pi/\rho)+1}
\end{equation}
and $P_{\rho,\mu}(r,\varphi,z)$ is defined by (\ref{eq:lemm_int0_P}).
\end{enumerate}
\end{corollary}

\begin{proof}
1) According to theorem~\ref{lemm:MLF_int0}, the~Mittag-Leffler function can be represented in the form (\ref{eq:MLF_int0}). This~representation is true for arbitrary $\delta_{1\rho}$ and $\delta_{2\rho}$ satisfying the conditions (\ref{eq:deltaCond_lemm_int0}). In~case if  $\delta_{1\rho}=\delta_{2\rho}=\delta_\rho$ the conditions (\ref{eq:deltaCond_lemm_int0}) take the form
\begin{equation}\label{eq:corol_Int0_deltaRho_case1}
  \frac{\pi}{2\rho}<\delta_\rho\leqslant\min\left(\pi,\frac{\pi}{\rho}\right)
\end{equation}
and the condition (\ref{eq:lem_CommonRepr_argZcond}) can be written in the form
\begin{equation}\label{eq:corol_Int0_argZ_case1}
  \frac{\pi}{2\rho}-\delta_\rho+\pi<\arg z<-\frac{\pi}{2\rho}+\delta_\rho+\pi.
\end{equation}
As a result, the~Mittag-Leffler function is written in the form
\begin{equation*}
  E_{\rho,\mu}(z)=\int_{1+\epsilon}^{\infty}K_{\rho,\mu}(r,-\delta_\rho,\delta_\rho,z)dr+ \int_{-\delta_\rho-\pi}^{\delta_\rho-\pi}P_{\rho,\mu}(1+\epsilon,\varphi,z)d\varphi.
\end{equation*}

We consider the integrand of the first integral and denote
\begin{equation}\label{eq:K_deltaRho_z}
K_{\rho,\mu}(r,\delta_\rho,z)\equiv K_{\rho,\mu}(r,-\delta_\rho,\delta_\rho,z).
\end{equation}
Using the definition $K_{\rho,\mu}(r,\varphi_1,\varphi_2,z)$ (see~(\ref{eq:K_lemm_int0})) we get
\begin{multline}\label{eq:corol_Int0_K_tmp0}
  K_{\rho,\mu}(r,\delta_\rho,z)= K_{\rho,\mu}(r,-\delta_\rho,\delta_\rho,z)\\
  = \frac{\rho}{2\pi i} \frac{(zr)^{\rho(1-\mu)}e^{-i\rho\pi(1-\mu)}\left(A_{\rho,\mu}(r,\delta_\rho,-\delta_\rho,z)- A_{\rho,\mu}(r,-\delta_\rho,\delta_\rho,z)\right)} {(r^2+2r\cos\delta_\rho+1)^2}.
\end{multline}
Using the definition $A_{\rho,\mu}(r,\omega_1,\omega_2,z)$ (see~(\ref{eq:A_lemm_int0})) and the fact that the function $\eta(r,\varphi,z)$ defined by~(\ref{eq:eta_lemm_int0}) is  an odd function according to the variable $\varphi$
\begin{equation*}
  \eta(r,-\varphi,z) =-\eta(r,\varphi,z),
  \end{equation*}
we have
\begin{multline*}
  A_{\rho,\mu}(r,\delta_\rho,-\delta_\rho,z)-A_{\rho,\mu}(r,-\delta_\rho,\delta_\rho,z)\\
  = \exp\left\{(zr)^\rho e^{-i\pi\rho}\cos(\rho\delta_\rho)\right\} (r^2+2r\cos(-\delta_\rho)+1)e^{i\eta(r,\delta_\rho,z)}(r+e^{i\delta_\rho}) \\
  -\exp\left\{(zr)^\rho e^{-i\pi\rho}\cos(-\rho\delta_\rho)\right\} (r^2+2r\cos\delta_\rho+1)e^{i\eta(r,-\delta_\rho,z)}(r+e^{-i\delta_\rho})\\
  =\exp\left\{(zr)^\rho e^{-i\pi\rho}\cos(\rho\delta_\rho)\right\} (r^2+2r\cos\delta_\rho+1)
  \left(re^{i\eta(r,\delta_\rho,z)}+e^{i(\eta(r,\delta_\rho,z)+\delta_\rho)}- re^{-i\eta(r,\delta_\rho,z)}-e^{-i(\eta(r,\delta_\rho,z)+\delta_\rho)}\right)\\
  =2i\exp\left\{(zr)^\rho e^{-i\pi\rho}\cos(\rho\delta_\rho)\right\} (r^2+2r\cos\delta_\rho+1) \left(r\sin(\eta(r,\delta_\rho,z))+\sin(\eta(r,\delta_\rho,z)+\delta_\rho)\right).
\end{multline*}
Now substituting this result in (\ref{eq:corol_Int0_K_tmp0}) we obtain (\ref{eq:K_coroll_int0}). Since in the proof process no additional restrictions on the values of parameters $\rho$, $\mu$ and on the value $\arg z$ were imposed, then the conditions for these parameters go from  theorem~\ref{lemm:MLF_int0} without change. Thus, we come to the conditions of the~corollary.

2) We consider the case $\delta_{1\rho}=\delta_{2\rho}=\pi/\rho$. As~we can see, the~case considered is a particular case of the previous one.  It follows from (\ref{eq:deltaCond_lemm_int0}) that this case can be implemented if $\rho\geqslant1$. For~the range of values $\arg z$ from (\ref{eq:corol_Int0_argZ_case1}) we get
\begin{equation*}
  -\frac{\pi}{2\rho}+\pi<\arg z<\frac{\pi}{2\rho}+\pi.
\end{equation*}

Now we consider the representation (\ref{eq:MLF_int0_deltaRho}). In~the case under consideration it will be written in the~form
\begin{equation*}
  E_{\rho,\mu}(z)=\int_{1+\epsilon}^{\infty}K_{\rho,\mu}\left(r,\pi/\rho,z\right)dr+ \int_{-\tfrac{\pi}{\rho}-\pi}^{\tfrac{\pi}{\rho}-\pi}P_{\rho,\mu}(1+\epsilon,\varphi,z)d\varphi.
\end{equation*}
We denote
\begin{equation}\label{eq:K_piRho_z}
K_{\rho,\mu}(r,z)\equiv K_{\rho,\mu}\left(r,\tfrac{\pi}{\rho},z\right).
\end{equation}
From (\ref{eq:eta_lemm_int0}) it follows that $\eta\left(r,\tfrac{\pi}{\rho},z\right)=(1-\mu)\pi$. Now using this result in (\ref{eq:K_coroll_int0}) we get (\ref{eq:K_piRho_coroll_int0}).
\end{proof}

From the proved corollary it follows that if the parameter values $\delta_{1\rho}$ and $\delta_{2\rho}$ coincide, then in this the kernel function $K_{\rho,\mu}(r,\varphi_1,\varphi_2,z)$ is significantly simplified. Recall that the parameters $\delta_{1\rho}$ and $\delta_{2\rho}$ are inclination angles of half-lines $S_1$ and $S_2$ in the contour $\gamma_\zeta$ relative to the axis of this contour (see~Figure~\ref{fig:loop_gammaZeta}). Since in the theorem~\ref{lemm:MLF_int} the axis of the contour  $\gamma_\zeta$ coincides with the real axis, then the selection of $\delta_{1\rho}=\delta_{2\rho}$ means that half-lines $S_1$ and $S_2$ run symmetrically in relation to the real axis. The~kernel function $K_{\rho,\mu}(r,\varphi_1,\varphi_2,z)$ takes the simplest form  in the case $\delta_{1\rho}=\delta_{2\rho}=\pi/\rho$.

Further, it is necessary for us to know the position of singular points of the integrand of the representation (\ref{eq:MLF_int}). This issue was studied in the work~\cite{Saenko2020a}. For~completeness of the statement here we give the result obtained in the work~\cite{Saenko2020a} and formulate it in the form of a~lemma.

\begin{lemma}\label{lemm:MLF_SingPoints}
For any real $\rho>1/2$ and any complex values of the parameter $\mu=\mu_R+i\mu_I$ the integrand of the representation (\ref{eq:MLF_int})
\begin{equation}\label{eq:Phi_fun}
  \Phi_{\rho,\mu}(\zeta,z)=\frac{\rho}{2\pi i}\frac{\exp\{(\zeta z)^\rho\}(\zeta z)^{\rho(1-\mu)}}{\zeta-1}.
\end{equation}
relative to the variable $\zeta$ has two singular points $\zeta=1$ and $\zeta=0$. The~point $\zeta=1$ is a pole of the first order. The~point  $\zeta=0$ is:
\begin{enumerate}
\item the regular point of the function $\Phi_{\rho,\mu}(\zeta,z)$, with~the values of parameters $\rho=n$, where $n=1,2,3,\dots$ (the positive integer), $\mu_I=0$ and $\mu_R=1-m_1/\rho$, where $m_1=0,1,2,3,\dots$ (the non-negative integer);
\item a pole of the order $m_2$, if~$\rho=n$, where $n=1,2,3,\dots$ (the positive integer), $\mu_I=0$ and~$\mu_R=1+m_2/\rho$, where $m_2=1,2,3,\dots$ (the positive integer);
\item the branch point, for~any other values of parameters $\rho, \mu_I, \mu_R$.
\end{enumerate}
\end{lemma}

The proof of this lemma can be found in the work~\cite{Saenko2020a}.

We will make the following remark to Corollary~\ref{coroll:MLF_int0_deltaRho}.
\begin{remark}
In corollary~\ref{coroll:MLF_int0_deltaRho} the special case under consideration $\delta_{1\rho}=\delta_{2\rho}=\pi/\rho$. We will assume that $\rho=1$ and study the behavior of the Formula (\ref{eq:MLF_int0_piRho}) in this case. As~a result, we obtain
\begin{equation}\label{eq:MLF_int0_piRho_rho=1}
  E_{1,\mu}(z)=\int_{1+\epsilon}^{\infty}K_{1,\mu}(r,z)dr+ \int_{-2\pi}^{0}P_{1,\mu}(1+\epsilon,\varphi,z)d\varphi.
\end{equation}
Using (\ref{eq:K_piRho_coroll_int0}) for $K_{1,\mu}(z)$ we obtain
\begin{multline*}
  K_{1,\mu}(r,z)=\frac{1}{\pi}(zre^{-i\pi})^{1-\mu} \exp\left\{-zre^{-i\pi} \right\}  \frac{r\sin(\pi(1-\mu))+\sin(\pi(1-\mu)+\pi)}{r^2-2r+1}=\\
  \frac{1}{\pi}(-z)^{1-\mu}e^{zr}\frac{\sin(\pi(1-\mu))(r-1)}{(r-1)^2}=
  \frac{1}{\pi}(-z)^{1-\mu}e^{zr}\frac{\sin(\pi(1-\mu))}{r-1}.
\end{multline*}
From here it is clear, if~$\mu=n$, where $n=0,\pm1,\pm2,\pm3,\dots$, then $\sin(\pi(1-\mu))=0$. Consequently,
\begin{equation}\label{eq:K(1,mu)}
  K_{1,\rho}(z)=\left\{\begin{array}{cc}
                  0, & \mu=n, \\
                  \frac{1}{\pi}(-z)^{1-\mu}e^{zr}\frac{\sin(\pi(1-\mu))}{r-1}, & \mu\neq n.
                \end{array}\right.
\end{equation}
Thus, with~integer values of $\mu$ the first summand in (\ref{eq:MLF_int0_piRho_rho=1}) becomes zero and to calculate the value of $E_{1,n}(z)$ it remains to calculate the second integral. To~calculate this integral, numerical methods can be used. However, in~this case, this integral can be calculated analytically using the residue~theory.

In fact, we return to the integral representation formulated in theorem~\ref{lemm:MLF_int}. Recall that we consider the case $\rho=1$. Using the notation (\ref{eq:Phi_fun}), the representation  (\ref{eq:MLF_int}) takes the form
\begin{equation}\label{eq:MLF_int_rho=1}
  E_{1,\mu}(z)=\int_{\gamma_\zeta}\Phi_{1,\mu}(\zeta,z)d\zeta,
\end{equation}
where the contour of integration $\gamma_\zeta$, defined by (\ref{eq:loop_gammaZeta}), is written in the form
\begin{equation}\label{eq:loop_gammaZeta_rho=1}
  \gamma_\zeta=\left\{\begin{array}{ll}
                       S_1=&\{\zeta: \arg\zeta=-2\pi,\quad |\zeta|\geqslant 1+\epsilon\},\\
                       C_\epsilon=&\{\zeta: -2\pi\leqslant\arg\zeta\leqslant0,\quad |\zeta|=1+\epsilon\},\\
                       S_2=&\{\zeta: \arg\zeta=0,\quad|\zeta|\geqslant1+\epsilon\}.
                     \end{array}\right.
\end{equation}

We represent the complex parameter $\mu$ in the form $\mu=\mu_R+i\mu_I$ and make use of  lemma~\ref{lemm:MLF_SingPoints}. According~to this lemma, the~function $\Phi_{1,\mu}(\zeta,z)$ at values  $\mu_I=0$  and $\mu_R=1-m_1$, where $m_1=0,1,2,3,\dots$ has one singular point $\zeta=1$ which is a pole of the first order. The~point $\zeta=0$, in~this case, is the regular point. In~case, if~$\mu_I=0$ and $\mu_R=1+m_2$, where $m_2=1,2,3,\dots$  the function $\Phi_{1,\mu}(\zeta,z)$  has two singular points: the point $\zeta=1$ is a pole of the first order and the point $\zeta=0$ is a pole of the order $m_2$. As~we can see, in~both cases the point $\zeta=0$ is not a branch point. As~a result, in~these two cases, the~function $\Phi_{1,\mu}(\zeta,z)$ is the entire function of a complex variable $\zeta$. From~here it follows that when $\mu_I=0$, and~$\mu_R=n$, where $n=0,\pm1,\pm2,\pm3,\dots$ the arc of the circle $C_\epsilon$ that enters the contour (\ref{eq:loop_gammaZeta_rho=1}) is the closed circle of radius $1+\epsilon$. The~half-lines $S_1$ and $S_2$ pass along the positive part of a real axis in mutually opposite directions. With~all other values of the parameter $\mu$ (when $\mu_I\neq0$ or $\mu_R\neq n$), according to lemma~\ref{lemm:MLF_SingPoints}, the~point $\zeta=0$ is a branch point of the function $\Phi_{1,\mu}(\zeta,z)$. In~this case, the~circle $C_\epsilon$ of the contour (\ref{eq:loop_gammaZeta_rho=1}) will not close up and half-lines $S_1$ and $S_2$ will go along the upper and lower banks of the cut of the complex plane which runs along the positive part of a real~axis.

It is clear from here that the result (\ref{eq:K(1,mu)}) is a consequence of  lemma~\ref{lemm:MLF_SingPoints} In fact, in~the case when the parameter $\mu$ takes integer real values, the~first and second items of lemma~\ref{lemm:MLF_SingPoints} turn out to be true. As~we have already pointed out, in~this case the arc of the circle $C_\epsilon$ of the contour (\ref{eq:loop_gammaZeta_rho=1}) is a closed circle and the half-lines $S_1$ and $S_2$ run along the positive part of a real axis in mutually opposite directions. Consequently, the~sum of the integrals along the half-lines $S_1$ and $S_2$ will be equal to zero. Next, it is necessary to recall that the improper integral in the expression (\ref{eq:MLF_int0_piRho_rho=1}) just corresponds to the sum of the integrals along the half-lines $S_1$ and $S_2$. Therefore, with~integer real values  $\mu$ it should be equal to zero which has been obtained. A~definite integral in (\ref{eq:MLF_int0_piRho_rho=1}) corresponds to integration along the closed circle. Therefore, one can use the theory of residues to calculate it.
\end{remark}

The calculation of the integral in (\ref{eq:MLF_int_rho=1}) using the theory of residues with integer real values of the parameter $\mu$ was carried out in the work~\cite{Saenko2020a}. For~completeness of the statement, we will give the results obtained in this paper and formulate them in the form of a corollary to Lemma~\ref{lemm:MLF_SingPoints}.

\begin{corollary}\label{coroll:MLF_case_rho=1}
For the values of the parameters $\rho=1$, $\delta_{1\rho}=\delta_{2\rho}=\pi$, any complex $z$,  satisfying the condition $\pi/2<\arg z<3\pi/2$ and for integer real values of the parameter $\mu=n, n=0,\pm1,\pm2,\pm3,\dots$ the Mittag-Leffler function has the~form:
\begin{enumerate}
\item if $n\leqslant1$ (i.e., $n=1,0,-1,-2,-3,\dots$), then
\begin{equation*}
E_{1,n}(z)=e^z z^{1-n},
\end{equation*}

\item if $n\geqslant2$ (i.e., $n=2,3,4,\dots$), then
\begin{equation*}
E_{1,n}(z)=z^{1-n}\left(e^z-\sum_{k=0}^{n-2}\frac{z^k}{k!}\right).
\end{equation*}
\end{enumerate}
\end{corollary}
The proof of this corollary can be found in the paper~\cite{Saenko2020a}.

\section{Integral Representation  ``B''}

The integral representation ``A'' consists of the sum of two integrals. As~it has been found earlier, the~improper integral in (\ref{eq:MLF_int0}) corresponds to the sum of integrals along the half-lines $S_1$ and $S_2$ of the contour $\gamma_\zeta$ in the representation (\ref{eq:MLF_int}), a~definite integral corresponds to the integral along the arc of the circle  $C_\epsilon$.  As~a result, in~analytical studies of the function $E_{\rho,\mu}(z)$, as~well as in the solution of problems where it is encountered, one should conduct studies of these two integrals. This causes certain difficulties. As~it will be shown below in the representation  (\ref{eq:MLF_int0}) one can get rid of an integral on the arc of the circle  $C_\epsilon$ and write the integral representation for the function $E_{\rho,\mu}(z)$ in the form of an improper integral. This representation will be much easier to use. However, as~a result of such a transition, some restrictions are imposed on the parameter values $\mu$.  The~integral representation of the Mittag-Leffler function represented in the following theorem will be called the representation ``B''.

\begin{theorem}[Representation ``B'']\label{lemm:MLF_int1}
For any real $\rho>1/2$ and any complex $\mu$ satisfying the condition $\Re\mu < 1+1/\rho$
for the function $E_{\rho,\mu}(z)$, the following integral representations are~valid:
\begin{enumerate}
  \item for any real $\delta_{1\rho}, \delta_{2\rho}$ satisfying the conditions  \begin{equation}\label{eq:deltaRhoCond_lem}
  \begin{array}{ll}
    \frac{\pi}{2\rho}<\delta_{1\rho}\leqslant\frac{\pi}{\rho},\quad \frac{\pi}{2\rho}<\delta_{2\rho}\leqslant\frac{\pi}{\rho}, &\quad \mbox{if}\quad\rho>1,  \\
     \frac{\pi}{2\rho}<\delta_{1\rho}<\pi,\quad \frac{\pi}{2\rho}<\delta_{2\rho}<\pi, &\quad \mbox{if}\quad 1/2<\rho\leqslant1,
  \end{array}
 \end{equation}
 and  any complex $z$ satisfying the condition
 $\frac{\pi}{2\rho}-\delta_{2\rho}+\pi<\arg z<-\frac{\pi}{2\rho}+\delta_{1\rho}+\pi$
 the Mittag-Leffler function can be represented in the form
\begin{equation}\label{eq:MLF_int1}
    E_{\rho,\mu}(z)=\int_{0}^{\infty}K_{\rho,\mu}(r,-\delta_{1\rho},\delta_{2\rho},z)dr,
  \end{equation}
 where $K_{\rho,\mu}(r,\varphi_1,\varphi_2,z)$ has the form (\ref{eq:K_lemm_int0});

  \item if  $1/2<\rho\leqslant1$ and $\delta_{1\rho}=\pi$, $\frac{\pi}{2\rho}<\delta_{2\rho}<\pi$, then for any complex $z$ satisfying the condition  $\frac{\pi}{2\rho}-\delta_{2\rho}+\pi<\arg z<-\frac{\pi}{2\rho}+2\pi$, the~Mittag-Leffler function can be represented in the form
  \begin{multline}\label{eq:lemm_MLF_int1_case2}
    E_{\rho,\mu}(z)=\int_{0}^{\infty}K_{\rho,\mu}^\prime(r,\delta_{2\rho},z)dr- \int_{0}^{1-\varepsilon_1}K_{\rho,\mu}^\prime(r,-\pi,z)dr -\\
    \int_{1+\varepsilon_1}^{\infty}K_{\rho,\mu}^\prime(r,-\pi,z)dr+
    \int_{-2\pi}^{-\pi}P_{\rho,\mu}^\prime(\varepsilon_1,\psi,-2,z)d\psi,
  \end{multline}
   where $\varepsilon_1$ is an arbitrary real number satisfying the condition $0<\varepsilon_1<1$,
  \begin{equation*}
    K_{\rho,\mu}^\prime(r, \varphi,z) =\frac{\rho}{2\pi i}\frac{\exp\left\{(zr)^\rho e^{-i\pi\rho}\cos(\rho\varphi)\right\}}{r^2+2r\cos\varphi+1}
    (zr)^{\rho(1-\mu)} e^{i[\eta(r,\varphi,z)-\pi\rho(1-\mu)]}\left(r+e^{i\varphi}\right),
  \end{equation*}
  where $\eta(r,\varphi,z)$ has the form (\ref{eq:eta_lemm_int0}) and
  \begin{multline*}
    P_{\rho,\mu}^\prime(\tau,\psi,k,z)=\frac{\rho\tau}{2\pi}\frac{\exp\left\{(z r(\tau,\psi))^\rho \cos(\rho\varphi(\tau,\psi,k))\right\} }{(r(\tau,\psi))^2-2 r(\tau,\psi) \cos(\varphi(\tau,\psi,k))+1}\times\\
    \times (zr(\tau,\psi))^{\rho(1-\mu)} e^{i[\chi^\prime(\tau,\psi,k,z)+\psi]}\left(r(\tau,\psi)
    e^{-i\varphi(\tau,\psi,k)}-1\right),
  \end{multline*}
  where
  \begin{align*}
    \chi^\prime(\tau,\psi,k,z) & =(z r(\tau,\psi))^\rho \sin(\rho\varphi(\tau,\psi,k)) +\rho(1-\mu)\varphi(\tau,\psi,k), \\
    r(\tau,\psi) & =\sqrt{\tau^2+2\tau\cos\psi+1},\\
    \varphi(\tau,\psi,k)&=\arctan\left(\frac{\tau\sin\psi}{\tau\cos\psi +1}\right)+k\pi;
  \end{align*}

  \item if $1/2<\rho\leqslant1$ and $\frac{\pi}{2\rho}<\delta_{1\rho}<\pi$, $\delta_{2\rho}=\pi$, then for any complex $z$ satisfying the condition $\frac{\pi}{2\rho}<\arg z<-\frac{\pi}{2\rho}+\delta_{1\rho}+\pi$, the~Mittag-Leffler function can be represented in the form
  \begin{multline}\label{eq:lemm_MLF_int1_case3}
    E_{\rho,\mu}(z)=\int_{0}^{1-\varepsilon_1}K_{\rho,\mu}^\prime(r,\pi,z)dr- \int_{1+\varepsilon_1}^{\infty}K_{\rho,\mu}^\prime(r,\pi,z)dr +\\
    +\int_{-\pi}^{0}P_{\rho,\mu}^\prime(\varepsilon_1,\psi,0,z)d\psi-
    \int_{0}^{\infty}K_{\rho,\mu}^\prime(r,-\delta_{1\rho},z)dr,
  \end{multline}
  where $\varepsilon_1$ is an arbitrary real number satisfying the condition $0<\varepsilon_1<1$;

  \item if $1/2<\rho\leqslant1$ and $\delta_{1\rho}=\delta_{2\rho}=\pi$, then for any complex $z$  satisfying the condition $\frac{\pi}{2\rho}<\arg z<-\frac{\pi}{2\rho}+2\pi$, the~Mittag-Leffler function can be represented in the form
  \begin{multline}\label{eq:lemm_MLF_int1_case4}
    E_{\rho,\mu}(z)=\int_{0}^{1-\varepsilon_1}K_{\rho,\mu}(r,\pi,z)dr-
    \int_{1+\varepsilon_1}^{\infty}K_{\rho,\mu}(r,\pi,z)dr +\\
    +\int_{-\pi}^{0}P_{\rho,\mu}^\prime(\varepsilon_1,\psi,0,z)d\psi+
    \int_{-2\pi}^{-\pi}P_{\rho,\mu}^\prime(\varepsilon_1,\psi,-2,z)d\psi,
  \end{multline}
  where $\varepsilon_1$ is an arbitrary real number satisfying the condition $0<\varepsilon_1<1$ and $K_{\rho,\mu}(r,\delta,z)$ has the form~(\ref{eq:K_coroll_int0}).
\end{enumerate}
\end{theorem}

\begin{proof}
The starting point of the proof is theorem~\ref{lemm:MLF_int} and the integral representation  (\ref{eq:MLF_int}) which is defined in it. In~view of the notation (\ref{eq:phiFun}), the~representation (\ref{eq:MLF_int}) will be written in the form
\begin{equation}\label{eq:lemm_MLF_int1_tmp0}
  E_{\rho,\mu}(z)=  \int_{\gamma_\zeta}\frac{\phi_{\rho,\mu}(\zeta,z)}{\zeta-1}d\zeta.
\end{equation}
The problem consists in calculating this contour~integral.

We consider an auxiliary integral
\begin{equation}\label{eq:intI_aux}
  I=\int_{\Gamma}\frac{\phi_{\rho,\mu}(\zeta,z)}{\zeta-1}d\zeta,
\end{equation}
where the contour $\Gamma$ (see~Figure~\ref{fig:loops_aux})  consists of the arc of the circle $C_\epsilon$ of radius $1+\epsilon$ with the center at the origin of coordinates, the~segment $\Gamma_2$, the~arc of the circle $C_\varepsilon$ the radius $\varepsilon$ with the center at the origin of coordinates and the segment $\Gamma_1$, which are defined in the following way:
\begin{equation}\label{eq:loop_Gamma_aux}
  \Gamma=\left\{\begin{array}{l}
                  C_\epsilon=\{\zeta:\ -\delta_{1\rho}-\pi\leqslant \arg\zeta \leqslant \delta_{2\rho}-\pi,\ |\zeta|=1+\epsilon\}, \\
                  \Gamma_2=\{\zeta:\ \arg\zeta=\delta_{2\rho}-\pi,\ 1+\epsilon \geqslant| \zeta|\geqslant \varepsilon\}, \\
                  C_\varepsilon=\{\zeta:\ \delta_{2\rho}-\pi\leqslant \arg\zeta \leqslant -\delta_{1\rho}-\pi,\ |\zeta|=\varepsilon\}, \\
                  \Gamma_1=\{\zeta:\ \arg\zeta=-\delta_{1\rho}-\pi,\ \varepsilon\leqslant|\zeta| \leqslant1+\epsilon\}.
                \end{array}\right.
\end{equation}
The contour is traversed in the positive direction. The~cut of the complex plane $\zeta$ goes along the positive part of a real~axis.

\begin{figure}[H]
  \centering
  \includegraphics[width=0.4\textwidth]{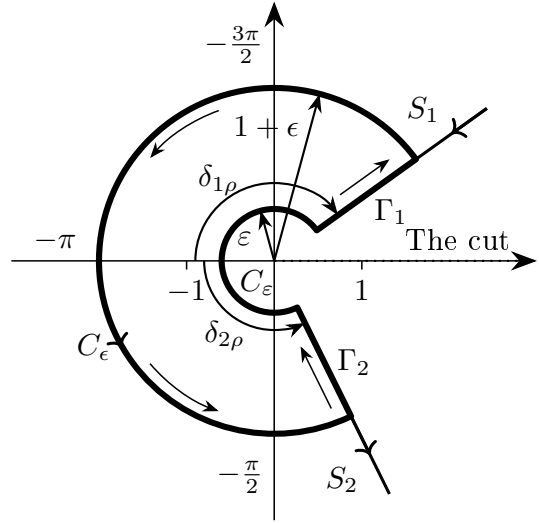}
  \caption{The auxiliary contour of integration $\Gamma$.}\label{fig:loops_aux}
\end{figure}

As it follows from lemma~\ref{lemm:MLF_SingPoints}, depending on a value of the parameter $\mu$, the integrand (\ref{eq:intI_aux}) has one or two poles. If~$\Re\mu\geqslant2$, then there are two poles in the points $\zeta=0$ and $\zeta=1$. If~$\Re\mu\leqslant1$, then there is one pole in the point $\zeta=1$. In~addition, at~non-integer values $\rho$ and $\mu$ the point $\zeta=0$ is the branch point. In~the case when $\frac{1}{2}<\rho\leqslant1$ and $\delta_{1\rho}=\pi$ the segment $\Gamma_1$ will pass through a singular point $\zeta=1$ (Figure~\ref{fig:loops_aux}). The~similar situation will be at $\frac{1}{2}<\rho\leqslant1$ and $\delta_{2\rho}=\pi$. In~this case, the segment $\Gamma_2$ will pass  through the point $\zeta=1$. These two cases will be considered separately. The~case $\frac{1}{2}<\rho\leqslant1$, $\delta_{1\rho}=\delta_{2\rho}=\pi$ also requires a separate consideration. In~this case each of the segments  $\Gamma_1$ and $\Gamma_2$ passes through a singular point $\zeta=1$. In~view of the foregoing, it is necessary to consider four~cases:
\begin{enumerate}
  \item
\begin{equation}\label{eq:deltaRhoCond_lem_proof}
  \begin{array}{ll}
    \frac{\pi}{2\rho}<\delta_{1\rho}\leqslant\frac{\pi}{\rho},\quad \frac{\pi}{2\rho}<\delta_{2\rho}\leqslant\frac{\pi}{\rho}, &\quad \mbox{if}\quad\rho>1,  \\
     \frac{\pi}{2\rho}<\delta_{1\rho}<\pi,\quad \frac{\pi}{2\rho}<\delta_{2\rho}<\pi, &\quad \mbox{if}\quad \frac{1}{2}<\rho\leqslant1,
  \end{array}
 \end{equation}
  \item $\frac{1}{2}<\rho\leqslant1$, $\delta_{1\rho}=\pi$, $\frac{\pi}{2\rho}<\delta_{2\rho}<\pi$,
  \item $\frac{1}{2}<\rho\leqslant1$, $\delta_{2\rho}=\pi$, $\frac{\pi}{2\rho}<\delta_{1\rho}<\pi$,
  \item $\frac{1}{2}<\rho\leqslant1$, $\delta_{1\rho}=\pi$, $\delta_{2\rho}=\pi$.
\end{enumerate}

\emph{Case 1.} We consider the first case at the beginning. Directly calculating the integral $I$ and substituting the variable of integration  $\zeta=r e^{i\varphi}$ we obtain
\begin{multline}
  I=\int_{C_\epsilon}\frac{\phi_{\rho,\mu}(\zeta,z)}{\zeta-1}d\zeta+\int_{\Gamma_2}\frac{\phi_{\rho,\mu}(\zeta,z)}{\zeta-1}d\zeta +\int_{C_\varepsilon}\frac{\phi_{\rho,\mu}(\zeta,z)}{\zeta-1}d\zeta+ \int_{\Gamma_1}\frac{\phi_{\rho,\mu}(\zeta,z)}{\zeta-1}d\zeta=\\
 =i\int\limits_{-\delta_{1\rho}-\pi}^{\delta_{2\rho}-\pi} \left. \frac{\phi_{\rho,\mu}(r e^{i\varphi},z)}{re^{i\varphi}-1}re^{i\varphi}d\varphi\right|_{r=1+\epsilon}+ \int\limits_{1+\epsilon}^\varepsilon\left.\frac{\phi_{\rho,\mu}(r e^{i\varphi},z)}{re^{i\varphi}-1} e^{i\varphi}dr \right|_{\varphi=\delta_{2\rho}-\pi}
  +i\int\limits_{\delta_{2\rho}-\pi}^{-\delta_{1\rho}-\pi} \left. \frac{\phi_{\rho,\mu}(r e^{i\varphi},z)}{re^{i\varphi}-1}re^{i\varphi}d\varphi\right|_{r=\varepsilon}\\
  + \int\limits_{\varepsilon}^{1+\epsilon}\left.\frac{\phi_{\rho,\mu}(r e^{i\varphi},z)}{re^{i\varphi}-1} e^{i\varphi}dr \right|_{\varphi=-\delta_{1\rho}-\pi}
  = i\hspace{-3mm}\int\limits_{-\delta_{1\rho}-\pi}^{\delta_{2\rho}-\pi}\hspace{-2mm} \frac{\phi_{\rho,\mu}((1+\epsilon) e^{i\varphi},z)}{(1+\epsilon)e^{i\varphi}-1}(1+\epsilon)e^{i\varphi}d\varphi
  +\int\limits_{1+\epsilon}^\varepsilon\frac{\phi_{\rho,\mu}(r e^{i(\delta_{2\rho}-\pi)},z)}{re^{i(\delta_{2\rho}-\pi)}-1} e^{i(\delta_{2\rho}-\pi)}dr \\
  + i\hspace{-3mm}\int\limits_{\delta_{2\rho}-\pi}^{-\delta_{1\rho}-\pi} \frac{\phi_{\rho,\mu}(\varepsilon e^{i\varphi},z)}{\varepsilon e^{i\varphi}-1}\varepsilon e^{i\varphi}d\varphi+
  \int\limits_{\varepsilon}^{1+\epsilon}\frac{\phi_{\rho,\mu}(r e^{i(-\delta_{1\rho}-\pi)},z)}{re^{i(-\delta_{1\rho}-\pi)}-1} e^{i(-\delta_{1\rho}-\pi)}dr
  =I_{C_\epsilon}+I_{\Gamma_2}+I_{C_\varepsilon}+I_{\Gamma_1}.\label{eq:intI_tmp1}
\end{multline}

Now we let $\varepsilon\to0$ in this expression and we will study the behavior of $I_{C_\varepsilon}$, $I_{\Gamma_1}$ and $I_{\Gamma_2}$.  We~consider the integral $I_{C_\varepsilon}$ at the beginning. For~this integral the relation is true
\begin{equation*}
  \lim_{\varepsilon\to0}|I_{C_\varepsilon}|\leqslant \lim_{\varepsilon\to0}\int\limits_{\delta_{2\rho}-\pi}^{-\delta_{1\rho}-\pi} \left|\frac{\phi_{\rho,\mu}(\varepsilon e^{i\varphi},z)}{\varepsilon e^{i\varphi}-1}\varepsilon e^{i\varphi}\right| |d\varphi|.
\end{equation*}
We consider the integrand. To~get rid of the complexity in the denominator we multiply and divide the integrand by $\varepsilon e^{-i\varphi}-1$. We also represent $z=|z|e^{i\arg z}$ and $\mu=\mu_R+i\mu_I$, and~for a power function we will use the representation $\xi^a=\exp\{a\ln\xi\}, \xi>0$. As~a result, we obtain
\begin{multline*}
  \lim_{\varepsilon\to0}\left|\frac{\phi_{\rho,\mu}(\varepsilon e^{i\varphi},z)}{\varepsilon e^{i\varphi}-1}\varepsilon e^{i\varphi}\right|=
  \frac{\rho}{2\pi}\lim_{\varepsilon\to0}\left|\frac{\exp\left\{\left(z \varepsilon e^{i\varphi}\right)^\rho\right\}\left(z\varepsilon e^{i\varphi}\right)^{\rho(1-\mu)}\varepsilon e^{i\varphi}\left(\varepsilon e^{-i\varphi}-1\right)} {\left(\varepsilon e^{i\varphi}-1\right)\left(\varepsilon e^{-i\varphi}-1\right)}\right|\\
  =\frac{\rho}{2\pi}\lim_{\varepsilon\to0}\left|\frac{\exp\left\{(|z|\varepsilon)^\rho e^{i\rho(\varphi+\arg z)}\right\}(|z|\varepsilon)^{\rho(1-\mu_R-i\mu_I)} e^{i\rho(1-\mu_R-i\mu_I)(\varphi+\arg z)}\varepsilon\left(\varepsilon-e^{i\varphi}\right)} {\varepsilon^2-2\varepsilon\cos\varphi+1}\right|\\
  =\frac{\rho}{2\pi}\lim_{\varepsilon\to0}\left|\frac{1}{\varepsilon^2-2\varepsilon\cos\varphi+1}\exp\left\{(|z|\varepsilon)^\rho e^{i\rho(\varphi+\arg z)}+ \rho(1-\mu_R-i\mu_I)\ln(|z|\varepsilon) +\right.\right.\\
   +\left.\left.i\rho(1-\mu_R-i\mu_I)(\varphi+\arg z)+ \ln\varepsilon\right\} (\varepsilon-e^{i\varphi})\right|\\
  =\frac{\rho}{2\pi}\lim_{\varepsilon\to0}\frac{1}{\varepsilon^2-2\varepsilon\cos\varphi+1}
  \left|\exp\left\{(|z|\varepsilon)^\rho\cos(\rho(\varphi+\arg z))+\rho(1-\mu_R)\ln(|z|\varepsilon)+\rho(\varphi+\arg z)\mu_I+\ln\varepsilon\right.\right.\\ \left.\left.+i\left[(|z|\varepsilon)^\rho\sin(\rho(\varphi+\arg z))+\rho(1-
  \mu_R)(\varphi+\arg z)-\rho\ln(|z|\varepsilon)\mu_I\right] \right\}\left(\varepsilon-e^{-i\varphi}\right)\right|\\
  =\frac{\rho}{2\pi}\lim_{\varepsilon\to0}\frac{\left|e^{A+iB}(\varepsilon-e^{i\varphi})\right|}{\varepsilon^2-2\varepsilon\cos\varphi+1},
\end{multline*}
where
\begin{align*}
  A =&(|z|\varepsilon)^\rho\cos(\rho(\varphi+\arg z))+\rho(1-\mu_R)\ln(|z|\varepsilon)+\rho(\varphi+\arg z)\mu_I+\ln\varepsilon, \\
  B =& (|z|\varepsilon)^\rho\sin(\rho(\varphi+\arg z))+\rho(1-\mu_R)(\varphi+\arg z)-\rho\mu_I\ln(|z|\varepsilon).
\end{align*}
For the numerator we have the estimate
\begin{equation*}
  \left|e^{A+iB}(\varepsilon-e^{i\varphi})\right|\leqslant \left|e^{A+iB}\varepsilon\right|+\left|e^{A+iB+i\varphi}\right|=e^A(\varepsilon+1).
\end{equation*}
Thus,
\begin{multline*}
  \frac{\rho}{2\pi}\lim_{\varepsilon\to0}\frac{\left|e^{A+iB}(\varepsilon-e^{i\varphi})\right|}{\varepsilon^2-2\varepsilon\cos\varphi+1}
  \leqslant \frac{\rho}{2\pi}\lim_{\varepsilon\to0}\frac{e^A(\varepsilon+1)}{\varepsilon^2-2\varepsilon\cos\varphi+1}\\
  =\frac{\rho}{2\pi}\lim_{\varepsilon\to0}\frac{1}{\varepsilon^2-2\varepsilon\cos\varphi+1}
  \exp\left\{(|z|\varepsilon)^\rho\cos(\rho(\varphi+\arg z))\right.\\
  +\left.\rho(1-\mu_R)\ln(|z|\varepsilon)+\rho(\varphi+\arg z)\mu_I+\ln\varepsilon +\ln(\varepsilon+1)\right\}\\
  =\frac{\rho}{2\pi}\lim_{\varepsilon\to0} \frac{1}{\varepsilon^2-2\varepsilon\cos\varphi+1} \exp\left\{(|z|\varepsilon)^\rho\cos(\rho(\varphi+\arg z))+\right.\\
  +\left.(\rho(1-\mu_R)+1)\ln\varepsilon +\ln(\varepsilon+1)+\rho(1-\mu_R)\ln(|z|)+\rho\mu_I(\varphi+\arg z)\right\} \\
  =\left\{\begin{array}{cc}
            \displaystyle 0, &\mu_R<1+\frac{1}{\rho}, \\
            \displaystyle \frac{\rho}{2\pi|z|}e^{\rho(\varphi+\arg z)\mu_I}, & \mu_R=1+\frac{1}{\rho}.
          \end{array}\right.
\end{multline*}
From this it follows that
\begin{equation}\label{eq:IntIC_varepsTo0}
  \lim_{\varepsilon\to0} I_{C_\varepsilon}=0,\quad \text{if}\quad \mu_R<1+1/\rho.
\end{equation}

Now we consider the behavior of the integral $I_{\Gamma_1}$ at $\varepsilon\to0$. We have the estimate
\begin{multline*}
  \lim_{\varepsilon\to0}\left|\int\limits_{\varepsilon}^{1+\epsilon}\frac{\phi_{\rho,\mu}(r e^{i(-\delta_{1\rho}-\pi)},z)}{re^{i(-\delta_{1\rho}-\pi)}-1} e^{i(-\delta_{1\rho}-\pi)}dr \right| \leqslant\lim_{\varepsilon\to0}\int\limits_{\varepsilon}^{1+\epsilon}\left|\frac{\phi_{\rho,\mu}(r e^{i(-\delta_{1\rho}-\pi)},z)}{re^{i(-\delta_{1\rho}-\pi)}-1} e^{i(-\delta_{1\rho}-\pi)}\right| |dr|\\
  =\int\limits_{0}^{1+\epsilon}\left|\frac{\phi_{\rho,\mu}(r e^{i(-\delta_{1\rho}-\pi)},z)}{re^{i(-\delta_{1\rho}-\pi)}-1} e^{i(-\delta_{1\rho}-\pi)}\right| |dr|.
\end{multline*}
Consequently, it is necessary to study the behavior of the integrand at $r\to0$.  For~convenience, we introduce the notation $\varphi_1=-\delta_{1\rho}-\pi$. Further, similarly to the previous case, we get rid of the complexity in the denominator. To~do this, we multiply and divide the integrand by the complex conjugate value of the denominator, i.e.,~by $(re^{-i\varphi_1}-1)$ and represent $z=|z|e^{i\arg z}$, $\mu=\mu_R+\mu_I$.  For~a power function we will make use of the representation $\xi^a=\exp\{a\ln\xi\}$. Using (\ref{eq:phiFun}) we get
\begin{multline*}
  \lim_{r\to0}\left|\frac{\phi_{\rho,\mu}(r e^{i(-\delta_{1\rho}-\pi)},z)}{re^{i(-\delta_{1\rho}-\pi)}-1} e^{i(-\delta_{1\rho}-\pi)}\right|
  =\frac{\rho}{2\pi}\lim_{r\to0}\left|\frac{\exp\left\{\left(zre^{i\varphi_1}\right)^\rho\right\} \left(zre^{i\varphi_1}\right)^{\rho(1-\mu)}e^{i\varphi_1}(re^{-i\varphi_1}-1)} {(re^{i\varphi_1}-1)(re^{-i\varphi_1}-1)} \right| \\
  =\frac{\rho}{2\pi}\lim_{r\to0}\left|\frac{1}{r^2-2r\cos\varphi_1+1} \exp\left\{(|z|r)^\rho e^{i\rho(\varphi_1+\arg z)}\right\}
  (|z|r)^{\rho(1-\mu_R+i\mu_I)}e^{i\rho(1-\mu_R-i\mu_I)(\varphi_1+\arg z)}(r-e^{i\varphi_1})\right|\\
  =\frac{\rho}{2\pi}\lim_{r\to0}\frac{1}{r^2-2r\cos\varphi_1+1}
  \left|\exp\left\{ (|z|r)^\rho\cos(\rho(\varphi_1+\arg z))+\rho\mu_I (\varphi_1+\arg z)+ \rho(1-\mu_R)\ln(|z|r) \right.\right.\\
  \left.\left.+i\left[(|z|r)^\rho\sin(\rho(\varphi_1+\arg z))+ \rho(1-\mu_R)(\varphi_1+\arg z)-\rho\mu_I\ln(|z|r)\right] \right\} (r-e^{i\varphi_1})\right|\\
  =\frac{\rho}{2\pi}\lim_{r\to0}\frac{\left|e^{A_1+iB_1}(r-e^{i\varphi_1})\right|}{r^2-2r\cos\varphi_1+1},
\end{multline*}
where
\begin{align*}
  A_1 =& (|z|r)^\rho\cos(\rho(\varphi_1+\arg z))+\rho\mu_I (\varphi_1+\arg z)+ \rho(1-\mu_R)\ln(|z|r), \\
  B_1 =& (|z|r)^\rho\sin(\rho(\varphi_1+\arg z))+ \rho(1-\mu_R)(\varphi_1+\arg z)-\rho\mu_I\ln(|z|r).
\end{align*}
For the numerator we have the estimate
\vspace{12pt}
\begin{equation*}
  \left|e^{A_1+iB_1}(r-e^{i\varphi_1})\right|\leqslant \left|re^{A_1+iB_1}\right|+\left|e^{A_1+iB_1+i\varphi_1}\right|=e^{A_1}(r+1).
\end{equation*}
Consequently, we obtain
\begin{multline*}
  \frac{\rho}{2\pi}\lim_{r\to0}\frac{\left|e^{A_1+iB_1}(r-e^{i\varphi_1})\right|}{r^2-2r\cos\varphi_1+1}
  \leqslant \frac{\rho}{2\pi} \lim_{r\to0}\frac{e^{A_1}(r+1)}{r^2-2r\cos\varphi_1+1} \\
  =\frac{\rho}{2\pi}\lim_{r\to0}\frac{1}{r^2-2r\cos\varphi_1+1} \exp\left\{(|z|r)^\rho\cos(\rho(\varphi_1+\arg z))+\rho\mu_I (\varphi_1+\arg z)\right.\\
  \left.+ \rho(1-\mu_R)\ln(|z|r)+\ln(r+1)\right\}
  = \left\{\begin{array}{cc}
             0, & \mu_R<1, \\
             \frac{\rho}{2\pi}e^{\rho(\varphi_1+\arg z)\mu_I}, &\mu_R=1.
           \end{array}\right.
\end{multline*}

We get the similar result for the integral $I_{\Gamma_2}$
\begin{equation*}
  \lim_{\varepsilon\to0}\left|\frac{\phi_{\rho,\mu}(r e^{i(\delta_{2\rho}-\pi)},z)}{re^{i(\delta_{2\rho}-\pi)}-1} e^{i(\delta_{2\rho}-\pi)}\right|=\left\{\begin{array}{cc}
             0, & \mu_R<1, \\
             \frac{\rho}{2\pi}e^{\rho(\delta_{2\rho}-\pi+\arg z)\Im\mu}, &\mu_R=1.
           \end{array}\right.
\end{equation*}

Thus, at~$\Re\mu\leqslant1$ the limits $\lim_{\varepsilon\to0} I_{\Gamma_1}$ and $\lim_{\varepsilon\to0} I_{\Gamma_2}$ will converge to the corresponding definite integrals
\begin{equation}\label{eq:int_Ig1_Ig2}
\begin{array}{l}
  \displaystyle\lim_{\varepsilon\to0}I_{\Gamma_1}=\int\limits_{0}^{1+\epsilon}\frac{\phi_{\rho,\mu}(r e^{i(-\delta_{1\rho}-\pi)},z)}{re^{i(-\delta_{1\rho}-\pi)}-1} e^{i(-\delta_{1\rho}-\pi)}dr,\\
  \displaystyle\lim_{\varepsilon\to0}I_{\Gamma_2}=-\int\limits_{0}^{1+\epsilon}\frac{\phi_{\rho,\mu}(r e^{i(\delta_{2\rho}-\pi)},z)}{re^{i(\delta_{2\rho}-\pi)}-1} e^{i(\delta_{2\rho}-\pi)}dr.
  \end{array}
\end{equation}

We will pay attention to the fact that in the case under consideration when the parameters $\rho,\delta_{1\rho}, \delta_{2\rho}$  satisfy the condition (\ref{eq:deltaRhoCond_lem_proof}), the~integrand of the integral $I$ (see~(\ref{eq:intI_aux})) inside the region limited by the contour  $\Gamma$, is an analytical function. Consequently, from~(\ref{eq:intI_tmp1}) we have
\begin{equation*}
  I=I_{C_\epsilon}+I_{\Gamma_2}+I_{C_\varepsilon}+I_{\Gamma_1}=0
\end{equation*}
Letting $\varepsilon\to0$ in this expression and taking into consideration (\ref{eq:IntIC_varepsTo0}) we obtain
\begin{equation}\label{eq:IC_epsilon_lemm_int1_tmp0}
 I_{C_\epsilon}=-\lim_{\varepsilon\to0}(I_{\Gamma_1}+I_{\Gamma_2}).
\end{equation}
Substituting here the expression for  $I_{C_\epsilon}$ and using  (\ref{eq:int_Ig1_Ig2}) we obtain
\begin{multline}\label{eq:I_C_vareps}
  i\int\limits_{-\delta_{1\rho}-\pi}^{\delta_{2\rho}-\pi} \frac{\phi_{\rho,\mu}((1+\epsilon) e^{i\varphi},z)}{(1+\epsilon)e^{i\varphi}-1}(1+\epsilon)e^{i\varphi}d\varphi\\
  =-\int\limits_{0}^{1+\epsilon}\frac{\phi_{\rho,\mu}(r e^{i(-\delta_{1\rho}-\pi)},z)}{re^{i(-\delta_{1\rho}-\pi)}-1} e^{i(-\delta_{1\rho}-\pi)}dr
  +\int\limits_{0}^{1+\epsilon}\frac{\phi_{\rho,\mu}(r e^{i(\delta_{2\rho}-\pi)},z)}{re^{i(\delta_{2\rho}-\pi)}-1} e^{i(\delta_{2\rho}-\pi)}dr,\quad \mu_R<1+\frac{1}{\rho}.
\end{multline}
Here it should be pointed out that at values $1<\mu_R<1+1/\rho$ the integrals (\ref{eq:int_Ig1_Ig2}) have a singularity in the point $r=0$ and at values $\mu_R\leqslant1$ the singularity~disappears.

Now let us return to the Mittag-Leffler function. Directly calculating (\ref{eq:lemm_MLF_int1_tmp0}) and replacing the variable of integration $\zeta=re^{i\varphi}$ we get
\begin{multline*}
  E_{\rho,\mu}(z)  =\int_{S_1}\frac{\phi_{\rho,\mu}(\zeta,z)}{\zeta-1}d\zeta +
  \int_{C_\epsilon}\frac{\phi_{\rho,\mu}(\zeta,z)}{\zeta-1}d\zeta
  +\int_{S_2}\frac{\phi_{\rho,\mu}(\zeta,z)}{\zeta-1}d\zeta
  =\int\limits_{\infty}^{1+\epsilon} \left.\frac{\phi_{\rho,\mu}\left(re^{i\varphi},z\right)} {re^{i\varphi}-1}e^{i\varphi}dr\right|_{\varphi=-\delta_{1\rho}-\pi}\\
  + i\hspace{-3mm}\int\limits_{-\delta_{1\rho}-\pi}^{\delta_{2\rho}-\pi} \left.\frac{\phi_{\rho,\mu}\left(re^{i\varphi},z\right)} {re^{i\varphi}-1}re^{i\varphi}\varphi\right|_{r=1+\epsilon}
  + \int\limits_{1+\epsilon}^{\infty} \left.\frac{\phi_{\rho,\mu}\left(re^{i\varphi},z\right)} {re^{i\varphi}-1}e^{i\varphi}dr\right|_{\varphi=\delta_{2\rho}-\pi}
  =\int\limits_{\infty}^{1+\epsilon} \frac{\phi_{\rho,\mu}\left(re^{i(-\delta_{1\rho}-\pi)},z\right)} {re^{i(-\delta_{1\rho}-\pi)}-1}e^{i(-\delta_{1\rho}-\pi)}dr\\
  + i\int\limits_{-\delta_{1\rho}-\pi}^{\delta_{2\rho}-\pi} \frac{\phi_{\rho,\mu}\left((1+\epsilon)e^{i\varphi},z\right)} {(1+\epsilon)e^{i\varphi}-1}(1+\epsilon)e^{i\varphi}\varphi
  + \int\limits_{1+\epsilon}^{\infty} \frac{\phi_{\rho,\mu}\left(re^{i(\delta_{2\rho}-\pi)},z\right)} {re^{i(\delta_{2\rho}-\pi)}-1}e^{i(\delta_{2\rho}-\pi)}dr.
\end{multline*}
Using here (\ref{eq:I_C_vareps}) we obtain
\begin{equation}\label{eq:MLF_tmp4}
  E_{\rho,\mu}(z)= \int\limits_{0}^{\infty}\left( \frac{\phi_{\rho,\mu}\left(re^{i(\delta_{2\rho}-\pi)},z\right)} {re^{i(\delta_{2\rho}-\pi)}-1}e^{i(\delta_{2\rho}-\pi)}
  -\frac{\phi_{\rho,\mu}\left(re^{i(-\delta_{1\rho}-\pi)},z\right)} {re^{i(-\delta_{1\rho}-\pi)}-1}e^{i(-\delta_{1\rho}-\pi)}\right)dr,\quad \mu_R<1+\frac{1}{\rho}.
\end{equation}

We consider the first summand in this expression. Getting rid of the complexity in the denominator and using (\ref{eq:phiFun}) we get
\begin{multline}\label{eq:lemm_int1_1stTerm}
  \frac{\phi_{\rho,\mu}\left(re^{i(\delta_{2\rho}-\pi)},z\right)} {re^{i(\delta_{2\rho}-\pi)}-1}e^{i(\delta_{2\rho}-\pi)}=
  \frac{\phi_{\rho,\mu}\left(re^{i(\delta_{2\rho}-\pi)},z\right)e^{i(\delta_{2\rho}-\pi)} \left(re^{-i(\delta_{2\rho}-\pi)}-1\right)} {\left(re^{i(\delta_{2\rho}-\pi)}-1)\right)\left(re^{-i(\delta_{2\rho}-\pi)}-1\right)} \\
  =\frac{\phi_{\rho,\mu}\left(re^{i(\delta_{2\rho}-\pi)},z\right) \left(r-e^{i(\delta_{2\rho}-\pi)}\right)}{r^2+2r\cos\delta_{2\rho}+1)}
  =\frac{\rho}{2\pi i}\frac{\exp\left\{\left(zre^{i(\delta_{2\rho}-\pi)}\right)^\rho \right\}\left(zr e^{i(\delta_{2\rho}-\pi)}\right)^{\rho(1-\mu)}\left(r+e^{i\delta_{2\rho}}\right)}{r^2+2r\cos\delta_{2\rho}+1}\\
  =\frac{\rho}{2\pi i}\frac{\exp\left\{(zr)^\rho e^{-i\rho\pi}(\cos(\rho\delta_{2\rho})+i\sin(\rho\delta_{2\rho}))\right\}(zr)^{\rho(1-\mu)}e^{i\rho(1-\mu)(\delta_{2\rho}-\pi)} \left(r+e^{i\delta_{2\rho}}\right)}{r^2+2r\cos\delta_{2\rho}+1}\\
  =\frac{\rho}{2\pi i}\frac{\exp\left\{(zr)^\rho e^{-i\rho\pi}\cos(\rho\delta_{2\rho})\right\}(zr)^{\rho(1-\mu)} e^{-i\rho(1-\mu)\pi} e^{i\eta(r,\delta_{2\rho},z)} \left(r+e^{i\delta_{2\rho}}\right)}{r^2+2r\cos\delta_{2\rho}+1},
\end{multline}
where $\eta(r,\varphi,z)$ has the form (\ref{eq:eta_lemm_int0}). Similarly, for~the second summand in (\ref{eq:MLF_tmp4}) we have
\begin{multline}\label{eq:lemm_int1_2ndTerm}
  \frac{\phi_{\rho,\mu}\left(re^{i(-\delta_{2\rho}-\pi)},z\right)} {re^{i(-\delta_{2\rho}-\pi)}-1}e^{i(-\delta_{2\rho}-\pi)}=\\
  = \frac{\rho}{2\pi i}\frac{\exp\left\{(zr)^\rho e^{-i\rho\pi}\cos(\rho\delta_{1\rho})\right\}(zr)^{\rho(1-\mu)} e^{-i\rho(1-\mu)\pi} e^{i\eta(r,-\delta_{1\rho},z)} \left(r+e^{-i\delta_{1\rho}}\right)}{r^2+2r\cos\delta_{1\rho}+1}
\end{multline}
Now substituting (\ref{eq:lemm_int1_1stTerm}) and (\ref{eq:lemm_int1_2ndTerm}) in (\ref{eq:MLF_tmp4}) and making simple transformations we finally obtain
\begin{equation*}
  E_{\rho,\mu}(z)=\int_{0}^{\infty} K_{\rho,\mu}(r,-\delta_{1\rho},\delta_{2\rho},z) dr,
\end{equation*}
where $K_{\rho,\mu}(r,\varphi_1,\varphi_2,z)$ is defined by the expression (\ref{eq:K_lemm_int0}).  The~values $\delta_{1\rho}$ and $\delta_{2\rho}$ satisfy the conditions~(\ref{eq:deltaRhoCond_lem_proof}). The~condition for the value $\arg z$ did not change and passes from theorem~\ref{lemm:MLF_int} without any changes. Thus, the~first item of the theorem is~proved.

\emph{Case~2.} Now we consider the case when $\frac{1}{2}<\rho\leqslant1$, $\delta_{1\rho}=\pi$ and $\pi/(2\rho)<\delta_{2\rho}<\pi$.  As~it was mentioned earlier, in~this case the segment $\Gamma_1$ of the auxiliary contour $\Gamma$ (see Figure~\ref{fig:loops_aux}) will go through the singular point $\zeta=1$. That is why, it is necessary to change the contour of integration in such a way that one could  bypass this point  leaving it outside the contour. The~auxiliary contour $\Gamma'$   that we get  will consist of the arc of the circle $C_\epsilon$ with the center at the origin of coordinates and radius $1+\epsilon, \epsilon>0$ (see Figure~\ref{fig:loop_Gamma1_aux_case2}), the~segment $AB$, the~circle $C_\varepsilon$ with the center at the origin of coordinates and radius $\varepsilon$ satisfying the condition $0<\varepsilon<1$ ), the~segment $CD$, the~arc of the circle $C'_{1\varepsilon_1}$ with the center in the point $\zeta'=1$ ($\arg\zeta'=-2\pi$) and radius $\varepsilon_1$ (here $\varepsilon_1<1-\varepsilon$ and $\varepsilon_1<\epsilon$)  and the segment $EF$.

For further study we need to parametrize this contour. As~a result, in~the plane $\zeta$ the arc of the circle $C_\epsilon$ can be written in the form $C_\epsilon=\{\zeta:\ -2\pi\leqslant \arg\zeta \leqslant \delta_{2\rho}-\pi,\ |\zeta|=1+\epsilon\}$. The~segment $AB$ is written in the form $AB=\{\zeta:\ \arg\zeta=\delta_{2\rho}-\pi,\ \varepsilon\leqslant|\zeta|\leqslant 1+\epsilon\}$. The~arc of the circle  $C_\varepsilon$ has the form $C_\varepsilon=\{\zeta:\ \delta_{2\rho}-\pi\leqslant\arg\zeta\leqslant-2\pi,\ |\zeta|=\varepsilon\}$. The~segments $CD$ and $EF$ are written in the form $CD=\{\zeta:\ \arg\zeta=-2\pi,\ \varepsilon\leqslant|\zeta|\leqslant 1-\varepsilon_1\}$ and $EF=\{\zeta:\ \arg\zeta=-2\pi,\ 1+\varepsilon_1\leqslant|\zeta|\leqslant 1+\epsilon\}$.

\begin{figure}[H]
  \centering
  \includegraphics[width=0.4\textwidth]{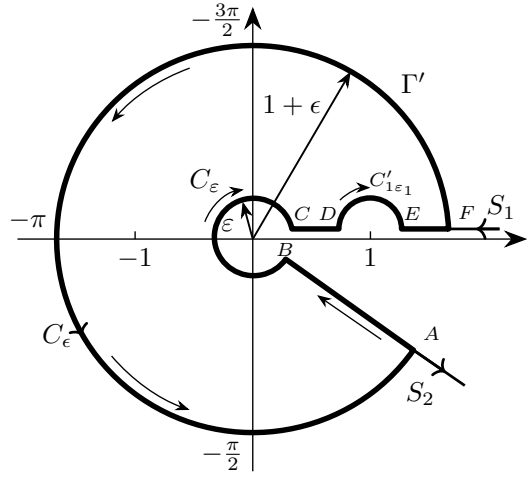}
  \caption{The auxiliary contour of integration $\Gamma'$.}\label{fig:loop_Gamma1_aux_case2}
\end{figure}

To parametrize the arc of the circle $C'_{1\varepsilon_1}$, we consider the mapping $u=\zeta-1$ of the complex plane $\zeta$ on the complex plane $u$. This mapping is a conformal mapping and is a left shift of the entire complex plane $\zeta$ by the value $1$. As~a result of such a shift, the point $\zeta=1$ in the plane $\zeta$ is mapped onto the point $u=0$ of the plane $u$. Thus, the~circle  $C'_{1\varepsilon_1}$ with the center in the point $\zeta=1$ and the radius $\varepsilon_1<1$ of the plane $\zeta$ will be mapped on the circle $C_{0\varepsilon_1}^u$ with the center in the point $u=0$ and the radius $\varepsilon_1$ on the plane $u$. Thus, to~parametrize the arc of the circle $C'_{1\varepsilon_1}$  in the plane $\zeta$ it is enough to parametrize the arc of the circle $C_{0\varepsilon_1}^u$ in the complex plane $u$ and then to map the complex plane $u$ on to the plane $\zeta$ with the help of an inverse conformal mapping $\zeta=u+1$.

We will represent the complex number $u$ in the form $u=\tau e^{i\psi}$. Then, an~arc of a circle $C_{0\varepsilon_1}^u$ in the complex plane  $u$ has the form $C_{0\varepsilon_1}^u=\{u:\ -2\pi\leqslant\psi\leqslant-\pi,\ \tau=\varepsilon_1\}$, where $\varepsilon_1\leqslant1$. Now we map the complex plane  $u$ onto the plane  $\zeta$. We have
$  \zeta=\tau e^{i\psi}+1= \tau\cos\psi+i\tau\sin\psi+1$. From~this we obtain
\begin{equation}\label{eq:maps_C1_vareps_tmp}
  |\zeta|=\sqrt{\tau^2+2\tau\cos\psi+1},\quad\arg\zeta=\arctan\left(\frac{\tau\sin\psi}{\tau\cos\psi+1}\right).
\end{equation}

Here it is necessary to note that $\arctan(x)$ is a multivalued function. The~principal branch of this function takes values in the interval $[-\pi/2,\pi/2]$. However, as~one can see from the definition of the auxiliary contour $\Gamma'$, the~center  of the arc of the circle $C'_{1\varepsilon_1}$ lies in the point $\zeta'=1$ in which $\arg\zeta'=-2\pi$. Therefore, it is necessary to choose the required branch in $\arctan(x)$ in such a way that the mapping (\ref{eq:maps_C1_vareps_tmp}) could map an arc of a circle $C_{0\varepsilon_1}^u$ of the plane $u$ on to an arc of a circle  $C'_{1\varepsilon_1}$ of the plane $\zeta$ with the center in the point $\zeta'=1$ and $\arg\zeta'=-2\pi$. As~a result, the~mapping (\ref{eq:maps_C1_vareps_tmp}) will take the form
\begin{equation}\label{eq:maps_C1_vareps}
  |\zeta|=\sqrt{\tau^2+2\tau\cos\psi+1},\quad\arg\zeta=\arctan\left(\frac{\tau\sin\psi}{\tau\cos\psi+1}\right) +k\pi.
\end{equation}
For the arc of the circle $C'_{1\varepsilon_1}$ we obtain $k=-2$. Here it should be pointed out that these formulas produce the mapping of the circle $C_{0\varepsilon_1}^u$ on to the circle $C'_{1\varepsilon_1}$ only in the case $\tau\leqslant1$. If~$\tau>1$, then these formulas will not be~true.

Thus, in~view of the remarks made the auxiliary contour $\Gamma'$ can be represented in the form
\begin{equation}\label{eq:loop_Gamma_aux_case2}
  \Gamma'=\left\{\begin{array}{l}
                  C_\epsilon=\{\zeta:\ -2\pi\leqslant \arg\zeta \leqslant \delta_{2\rho}-\pi,\ |\zeta|=1+\epsilon\}, \\
                  AB=\{\zeta:\ \arg\zeta=\delta_{2\rho}-\pi,\ \varepsilon\leqslant|\zeta|\leqslant 1+\epsilon\}, \\
                  C_\varepsilon=\{\zeta:\ \delta_{2\rho}-\pi\leqslant\arg\zeta\leqslant-2\pi,\ |\zeta|=\varepsilon\}, \\
                  CD=\{\zeta:\ \arg\zeta=-2\pi,\ \varepsilon\leqslant|\zeta|\leqslant 1-\varepsilon_1\},\\
                  C'_{1\varepsilon_1}=\left\{\zeta:\ \begin{array}{c}
                                                 \arg\zeta=\arctan\left(\frac{\tau\sin\psi}{\tau\cos\psi+1}\right) +k\pi,\\
                                                 |\zeta|=\sqrt{\tau^2+2\tau\cos\psi+1},\\
                                                 -\pi\geqslant\psi\geqslant-2\pi,\ \tau=\varepsilon_1,\ k=-2,
                                               \end{array}  \right\}\\
                  EF=\{\zeta:\ \arg\zeta=-2\pi,\ 1+\varepsilon_1\leqslant|\zeta|\leqslant 1+\epsilon\}.
                \end{array}\right.
\end{equation}
The contour is traversed in a positive~direction.

Next, we consider an auxiliary integral
\begin{equation*}
  I'=\int_{\Gamma'}\frac{\phi_{\rho,\mu}(\zeta,z)}{\zeta-1}d\zeta.
\end{equation*}
By calculating this integral we have
\begin{multline}\label{eq:I'_case2}
  I'= \int_{\Gamma'}\frac{\phi_{\rho,\mu}(\zeta,z)}{\zeta-1}d\zeta=\int_{C_\epsilon}\frac{\phi_{\rho,\mu}(\zeta,z)}{\zeta-1}d\zeta+
  \int_{AB}\frac{\phi_{\rho,\mu}(\zeta,z)}{\zeta-1}d\zeta+ \int_{C_\varepsilon}\frac{\phi_{\rho,\mu}(\zeta,z)}{\zeta-1}d\zeta +  \int_{CD}\frac{\phi_{\rho,\mu}(\zeta,z)}{\zeta-1}d\zeta + \\ \int_{C'_{1\varepsilon_1}}\frac{\phi_{\rho,\mu}(\zeta,z)}{\zeta-1}d\zeta + \int_{EF}\frac{\phi_{\rho,\mu}(\zeta,z)}{\zeta-1}d\zeta
  =I_{C_\epsilon}+I'_{AB}+I_{C_\varepsilon}+I'_{CD}+ I'_{C'_{1\varepsilon_1}}+I'_{EF}.
\end{multline}
One should pay attention to the fact that inside the region limited by the contour $\Gamma'$ the integrand $\phi_{\rho,\mu}(\zeta,z)/(\zeta-1)$ of the integral $I'$ is an analytical function of the complex variable $\zeta$. Consequently, according to the Cauchy integral theorem
\begin{equation*}
  I_{C_\epsilon}+I'_{AB}+I_{C_\varepsilon}+I'_{CD}+ I'_{C'_{1\varepsilon_1}}+I'_{EF}=0.
\end{equation*}
From here we get that
\begin{equation}\label{eq:lemm_MLF_int1_ICeps_summ}
  I_{C_\epsilon}=-I'_{AB}-I_{C_\varepsilon}-I'_{CD}- I'_{C'_{1\varepsilon_1}}-I'_{EF}.
\end{equation}
Thus, it is possible to substitute integration on the arc of the circle $C_\epsilon$ by integration on the remaining parts of the contour $\Gamma'$.

We consider each integral in the right part of this expression.  The~integral $I_{C_{\varepsilon}}$ has already been considered by us earlier in this lemma. Using (\ref{eq:IntIC_varepsTo0}) we obtain that in the case considered
\begin{equation}\label{eq:lemma6_I_Cvareps_case2}
  \lim_{\varepsilon\to0}I_{C_\varepsilon}=0,\quad \mu_R<1+1/\rho.
\end{equation}

Consider the integral $I'_{AB}$. Representing the complex number $\zeta$ in the form $\zeta=re^{i\varphi}$ and using~$(\ref{eq:loop_Gamma_aux_case2})$ we obtain
\begin{equation}\label{eq:I_AB_case2}
  I'_{AB}=\int_{AB}\frac{\phi_{\rho,\mu}(\zeta,z)}{\zeta-1}d\zeta=
  \int_{1+\epsilon}^{\varepsilon} \left.\frac{\phi_{\rho,\mu}\left(re^{i\varphi},z\right)}{re^{i\varphi}-1} e^{i\varphi} dr\right|_{\varphi=\delta_{2\rho}-\pi}
  = \int_{1+\epsilon}^{\varepsilon} \frac{\phi_{\rho,\mu}\left(re^{i(\delta_{2\rho}-\pi)},z\right)}{re^{i(\delta_{2\rho}-\pi)}-1}e^{i(\delta_{2\rho}-\pi)}dr.
\end{equation}
We introduce the notation
\begin{equation}\label{eq:lemm_MLF_int1_K'_def}
  \frac{\phi_{\rho,\mu}\left(re^{i(\delta-\pi)},z\right)}{re^{i(\delta-\pi)}-1}e^{i(\delta-\pi)}=K^\prime_{\rho,\mu}(r,\delta,z).
\end{equation}
Now we transform the function $K^\prime_{\rho,\mu}(r,\delta,z)$. Using the definition for $\phi_{\rho,\mu}(\zeta,z)$ (see~(\ref{eq:phiFun})) and getting rid of the complexity in the denominator we obtain
\begin{multline*}
  K^\prime_{\rho,\mu}(r,\delta,z)=
  \frac{\rho}{2\pi i} \frac{\exp\left\{ \left(zr e^{i(\delta-\pi)}\right)^\rho\right\} \left(zr e^{i(\delta-\pi)}\right)^{\rho(1-\mu)} e^{i(\delta-\pi)} \left(re^{-i(\delta-\pi)}-1\right)} {\left(re^{i(\delta-\pi)}-1\right) \left(re^{-i(\delta-\pi)}-1\right)}=\\
  \frac{\rho}{2\pi i}\frac{\exp\left\{ (zr)^\rho e^{-i\rho\pi}[\cos(\rho\delta)+i\sin(\rho\delta)]\right\} \left( zr e^{-i\pi}\right)^{\rho(1-\mu)} e^{i\rho(1-\mu)\delta}\left(r-e^{i(\delta-\pi)}\right)}{r^2-2r\cos(\delta-\pi)+1}=\\
  \frac{\rho}{2\pi i} \frac{\exp\left\{(zr)^\rho e^{-i\rho\pi}\cos(\rho\delta)\right\} \left(zr e^{-i\pi}\right)^{\rho(1-\mu)} }{r^2+2r\cos\delta+1}
  \exp\left\{i\left[(zr)^\rho e^{-i\pi\rho}\sin(\rho\delta)+\rho(1-\mu)\delta\right]\right\} \left(r+e^{i\delta}\right)
\end{multline*}
By introducing the notation $\eta(r,\varphi,z)=(zr)^\rho e^{-i\pi\rho}\sin(\rho\delta)+\rho(1-\mu)\delta$,
that coincides with (\ref{eq:eta_lemm_int0}) we obtain that $K_{\rho,\mu}^\prime(r,\delta,z)$ can be represented in the form
\begin{equation}\label{eq:K'_case2}
  K^\prime_{\rho,\mu}(r,\delta,z)= \frac{\rho}{2\pi i} \frac{\exp\left\{(zr)^\rho e^{-i\rho\pi}\cos(\rho\delta)\right\} \left(zr e^{-i\pi}\right)^{\rho(1-\mu)} e^{i\eta(r,\varphi,z)} \left(r+e^{i\delta}\right)}{r^2+2r\cos\delta+1}.
\end{equation}
Thus, for~the integral $I'_{AB}$ we get
\begin{equation}\label{eq:lemma6_I_AB_case2}
  I'_{AB}=\int_{1+\epsilon}^{\varepsilon}K_{\rho,\mu}^\prime(r,\delta_{2\rho},z)dr.
\end{equation}

Now consider the integral $I'_{CD}$. Representing the complex number $\zeta$ in the form $\zeta=re^{i\varphi}$ and using (\ref{eq:loop_Gamma_aux_case2}) we get
\begin{equation*}
  I'_{CD}=\int_{CD}\frac{\phi_{\rho,\mu}(\zeta,z)}{\zeta-1}d\zeta=
  \int_{\varepsilon}^{1-\varepsilon_1} \left.\frac{\phi_{\rho,\mu}\left(re^{i\varphi},z\right)}{re^{i\varphi}-1} e^{i\varphi} dr\right|_{\varphi=-\pi-\pi}
  =\int_{\varepsilon}^{1-\varepsilon_1} \frac{\phi_{\rho,\mu}\left(re^{i(-\pi-\pi)},z\right)}{re^{i(-\pi-\pi)}-1}e^{i(-\pi-\pi)}dr.
\end{equation*}
Comparing this expression with (\ref{eq:I_AB_case2}), we notice that the integrand of the integral obtained  is similar to the integrand  of the integral (\ref{eq:I_AB_case2}). Therefore, using (\ref{eq:lemm_MLF_int1_K'_def}) we get that the integrand can be represented in the form (\ref{eq:K'_case2}). Thus, for~the integral $I'_{CD}$ we obtain
\begin{equation}\label{eq:lemma6_I_CD_case2}
  I'_{CD}=\int_{\varepsilon}^{1-\varepsilon_1} K_{\rho,\mu}^\prime(r,-\pi,z)dr.
\end{equation}

Similarly, using (\ref{eq:loop_Gamma_aux_case2}) for the integral $I'_{EF}$ we obtain
\begin{equation*}
  I'_{EF}=\int_{EF}\frac{\phi_{\rho,\mu}(\zeta,z)}{\zeta-1}d\zeta=
  \int_{1+\varepsilon_1}^{1+\epsilon} \left.\frac{\phi_{\rho,\mu}\left(re^{i\varphi},z\right)}{re^{i\varphi}-1} e^{i\varphi} dr\right|_{\varphi=-\pi-\pi}
  =\int_{1+\varepsilon_1}^{1+\epsilon} \frac{\phi_{\rho,\mu}\left(re^{i(-\pi-\pi)},z\right)}{re^{i(-\pi-\pi)}-1}e^{i(-\pi-\pi)}dr.
\end{equation*}
Using now (\ref{eq:lemm_MLF_int1_K'_def}) and (\ref{eq:K'_case2}) the integral  $I'_{EF}$ takes the form
\begin{equation}\label{eq:lemma6_I_EF_case2}
  I'_{EF}=\int_{1+\varepsilon_1}^{1+\epsilon}  K_{\rho,\mu}^\prime(r,-\pi,z)dr.
\end{equation}

In the sum (\ref{eq:I'_case2}) it remains to consider the integral $I'_{C'_{1\varepsilon_1}}$. This integral is taken along the arc of the circle $C'_{1\varepsilon_1}$ with the center in the point $\zeta=1$ (at that $\arg\zeta=-2\pi$) and radius $\varepsilon_1$.   The~contour is traversed in the direction from the point  $D$ to the point $E$ (see Figure~\ref{fig:loop_Gamma1_aux_case2}). As~it was shown earlier, the arc of the  circle $C'_{1\varepsilon_1}$ in the plane $\zeta$ can be given in the form (see~(\ref{eq:loop_Gamma_aux_case2}))
\vspace{12pt}
\begin{equation}\label{eq:lemm_MLF_int1_C1vareps1_case2}
  C'_{1\varepsilon_1}=\left\{\zeta:\
  \begin{array}{c}
  \arg\zeta=\arctan\left(\frac{\tau\sin\psi}{\tau\cos\psi+1}\right) +k\pi,
  |\zeta|=\sqrt{\tau^2+2\tau\cos\psi+1},\\
   -\pi\geqslant\psi\geqslant-2\pi,\ \tau=\varepsilon_1,\ k=-2
   \end{array} \right\}
\end{equation}
As one can see, the~arc traverse in the direction from the point $D$ to the point $E$ corresponds to a change of the parameter $\psi$ from $-\pi$ to $-2\pi$. Thus, in~the case under consideration the contour $C'_{1\varepsilon_1}$ on the complex plane $\zeta$ can be represented in the form
\begin{equation}\label{eq:zeta_I_C1vareps}
\zeta=r(\varepsilon_1,\psi) e^{i\varphi(\varepsilon_1,\psi,-2)},
\end{equation}
where $-\pi\geqslant\psi\geqslant-2\pi$, $0<\varepsilon_1<1$ and
\begin{equation}\label{eq:lemm_MLF_r_varphi_C1vareps1}
\begin{array}{l}
  r(\tau,\psi)=\sqrt{\tau^2+2\tau\cos\psi+1}, \\ \varphi(\tau,\psi,k)=\arctan\left(\frac{\tau\sin\psi}{\tau\cos\psi+1}\right) +k\pi
  \end{array}
\end{equation}

Representing now in the integral $I'_{C'_{1\varepsilon_1}}$ the complex number $\zeta$ in the form (\ref{eq:zeta_I_C1vareps}) we find
\begin{equation}\label{eq:I_C1vareps1_case2_tmp0}
  I'_{C'_{1\varepsilon_1}}=\int_{C'_{1\varepsilon_1}}\frac{\phi_{\rho,\mu}(\zeta,z)}{\zeta-1}d\zeta=
  -i\varepsilon_1\int_{-2\pi}^{-\pi}\frac{\phi_{\rho,\mu}\left(r(\varepsilon_1,\psi) e^{i\varphi(\varepsilon_1,\psi,-2)},z\right)} {r(\varepsilon_1,\psi) e^{i\varphi(\varepsilon_1,\psi,-2)}-1} e^{i\psi}d\psi.
\end{equation}
One should pay attention that traversing along the arc of the circle $C'_{1\varepsilon_1}$ in the direction from the point $D$ to the point $E$ (see~Figure~\ref{fig:loop_Gamma1_aux_case2}) corresponds to the negative  traversing direction along the contour. Here the value of the parameter $\psi=-\pi$ corresponds to the point $D$ and the value of the parameter $\psi=-2\pi$ to the point $E$. That is why in the integral derived one must transpose the limits of~integration.

We now consider the integrand of this integral. We introduce the following notation for it
\begin{equation}\label{eq:lemm_MLF_int1_P'_case2}
  i\tau\frac{\phi_{\rho,\mu}\left(r(\tau,\psi) e^{i\varphi(\tau,\psi,k)},z\right)} {r(\tau,\psi) e^{i\varphi(\tau,\psi,k)}-1} e^{i\psi}=P_{\rho,\mu}^\prime(\tau,\psi,k,z).
\end{equation}
We transform this expression using (\ref{eq:phiFun}) and getting rid of the complexity in the denominator we~get
\begin{multline}\label{eq:lemm_MLF_int1_P'_def_case}
  P_{\rho,\mu}^\prime(\tau,\psi,k,z)=\frac{\rho\tau}{2\pi} \frac{\exp\left\{\left(z r(\tau,\psi) e^{i\varphi(\tau,\psi,k)}\right)^\rho\right\}\left(zr(\tau,\psi) e^{i\varphi(\tau,\psi,k)}\right)^{\rho(1-\mu)}e^{i\psi}}{r(\tau,\psi) e^{i\varphi(\tau,\psi,k)}-1}\\
  =\frac{\rho\tau}{2\pi}\frac{\exp\left\{\left(z r(\tau,\psi)\right)^\rho(\cos(\rho\varphi(\tau,\psi,k))+ i\sin(\rho\varphi(\tau,\psi,k)))\right\} }{\left(r(\tau,\psi)e^{i\varphi(\tau,\psi,k)}-1\right) \left(r(\tau,\psi)e^{-i\varphi(\tau,\psi,k)}-1\right)}\times\\\times \left(z r(\tau,\psi)\right)^{\rho(1-\mu)} e^{i\rho(1-\mu)\varphi(\tau,\psi,k)} \left(r(\tau,\psi)e^{-i\varphi(\tau,\psi,k)}-1\right) e^{i\psi}\\
  =\frac{\rho\tau}{2\pi} \frac{\exp\left\{(zr(\tau,\psi))^\rho\cos(\rho\varphi(\tau,\psi,k))\right\} (zr(\tau,\varphi))^{\rho(1-\mu)}}{(r(\tau,\psi))^2-2r(\tau,\psi)\cos\varphi(\tau,\psi,k)+1}\times\\
  \times \exp\left\{i[(zr(\tau,\psi))^\rho\sin(\rho\varphi(\tau,\psi,k))+ \rho(1-\mu)\varphi(\tau,\psi,k)]\right\}\left(r(\tau,\psi)e^{-i\varphi(\tau,\psi,k)}-1\right) e^{i\psi}\\
  =\frac{\rho\tau}{2\pi} \frac{\exp\left\{(zr(\tau,\psi))^\rho\cos(\rho\varphi(\tau,\psi,k))\right\}(zr(\tau,\varphi))^{\rho(1-\mu)}} {(r(\tau,\psi))^2-2r(\tau,\psi)\cos\varphi(\tau,\psi,k)+1}
   e^{i[\chi'(\tau,\psi,k,z)+\psi]} \left(r(\tau,\psi)e^{-i\varphi(\tau,\psi,k)}-1\right),
\end{multline}
where
\begin{equation*}
  \chi'(\tau,\psi,k,z)=(zr(\tau,\psi))^\rho\sin(\rho\varphi(\tau,\psi,k))+ \rho(1-\mu)\varphi(\tau,\psi,k).
\end{equation*}
As a result the integral (\ref{eq:I_C1vareps1_case2_tmp0}) takes the form
\begin{equation}\label{eq:lemma6_I_C1_vareps1_case2}
  I'_{C'_{1\varepsilon_1}}=-\int_{-2\pi}^{-\pi}P_{\rho,\mu}^\prime(\varepsilon_1,\psi,-2,z)d\psi.
\end{equation}

Now we get back to the expression (\ref{eq:lemm_MLF_int1_ICeps_summ}) and let $\varepsilon\to0$ in this expression. It is necessary to point out that the integrals $I_{C_\epsilon}$, $I'_{C'_{1\varepsilon_1}}$ and $I'_{EF}$ do not depend on  $\varepsilon$ and, consequently, they will not change  with such a passage to the limit. In~view of this, we have
\begin{equation*}
  I_{C_\epsilon}=-\lim_{\varepsilon\to0} I'_{AB}- \lim_{\varepsilon\to0}I_{C_\varepsilon}-\lim_{\varepsilon\to0}I'_{CD}- I'_{C'_{1\varepsilon_1}}-I'_{EF}.
\end{equation*}
Taking into account (\ref{eq:lemma6_I_Cvareps_case2}) and using the expressions (\ref{eq:lemma6_I_AB_case2}), (\ref{eq:lemma6_I_CD_case2}), (\ref{eq:lemma6_I_EF_case2}) and (\ref{eq:lemma6_I_C1_vareps1_case2}) we get
\begin{multline*}
  I_{C_\epsilon}=-\lim_{\varepsilon\to0}\int_{1+\epsilon}^{\varepsilon}K_{\rho,\mu}^\prime(r,\delta_{2\rho},z)dr-
  \lim_{\varepsilon\to0}\int_{\varepsilon}^{1-\varepsilon_1} K_{\rho,\mu}^\prime(r,-\pi,z)dr\\
  +\int_{-2\pi}^{-\pi}P_{\rho,\mu}^\prime(\varepsilon_1,\psi,-2,z)d\psi-
  \int_{1+\varepsilon_1}^{1+\epsilon}  K_{\rho,\mu}^\prime(r,-\pi,z)dr,\quad \Re\mu<1+1/\rho.
\end{multline*}
From this we obtain
\begin{multline}\label{eq:lemm_MLF_int1_I_Ceps_case2}
  I_{C_\epsilon}=-\int_{1+\epsilon}^{0}K_{\rho,\mu}^\prime(r,\delta_{2\rho},z)dr-
  \int_{0}^{1-\varepsilon_1} K_{\rho,\mu}^\prime(r,-\pi,z)dr\\
  +\int_{-2\pi}^{-\pi}P_{\rho,\mu}^\prime(\varepsilon_1,\psi,-2,z)d\psi-
  \int_{1+\varepsilon_1}^{1+\epsilon}  K_{\rho,\mu}^\prime(r,-\pi,z)dr,\quad \Re\mu<1+1/\rho.
\end{multline}

Now we get back to the Mittag-Leffler function.  Assuming that $\delta_{1\rho}=\pi$ the contour of integration $\gamma_\zeta$ in (\ref{eq:lemm_MLF_int1_tmp0}) takes the form
\begin{equation*}
  \gamma_\zeta=\left\{\begin{array}{l}
                  S_1=\{\zeta:\ \arg\zeta=-2\pi,\ |\zeta|\geqslant 1+\epsilon\}\\
                  C_\epsilon=\{\zeta:\ -2\pi\leqslant \arg\zeta \leqslant \delta_{2\rho}-\pi,\ |\zeta|=1+\epsilon\} \\
                  S_2=\{\zeta:\ \arg\zeta=\delta_{2\rho}-\pi,\ |\zeta|\geqslant 1+\epsilon\}
                \end{array}\right.
\end{equation*}
By calculating (\ref{eq:lemm_MLF_int1_tmp0}) directly and representing a complex number $\zeta$ in the form $\zeta=re^{i\varphi}$ we obtain
\begin{multline*}
  E_{\rho,\mu}(z)=\int_{S_1}\frac{\phi_{\rho,\mu}(\zeta,z)}{\zeta-1}d\zeta +
  \int_{C_{\epsilon}} \frac{\phi_{\rho,\mu}(\zeta,z)}{\zeta-1}d\zeta + \int_{S_1}\frac{\phi_{\rho,\mu}(\zeta,z)}{\zeta-1}d\zeta\\
  = \int_{\infty}^{1+\epsilon}\left.\frac{\phi_{\rho,\mu}\left(re^{i\varphi},z\right)}{re^{i\varphi}-1} e^{i\varphi}dr\right|_{\varphi=-\pi-\pi} + I_{C_\epsilon}
  +\int_{1+\epsilon}^{\infty}\left.\frac{\phi_{\rho,\mu}\left(re^{i\varphi},z\right)}{re^{i\varphi}-1}e^{i\varphi}dr \right|_{\varphi=\delta_{2\rho}-\pi}\\
  =\int_{\infty}^{1+\epsilon}\frac{\phi_{\rho,\mu}\left(re^{i(-\pi-\pi)},z\right)}{re^{i(-\pi-\pi)}-1} e^{i(-\pi-\pi)}dr + I_{C_\epsilon}
  +\int_{1+\epsilon}^{\infty}\frac{\phi_{\rho,\mu}\left(re^{i(\delta_{2\rho}-\pi)},z\right)}{re^{i(\delta_{2\rho}-\pi)}-1} e^{i(\delta_{2\rho}-\pi)}dr
\end{multline*}
using the notation (\ref{eq:lemm_MLF_int1_K'_def}), this expression can be written in the form
\begin{equation*}
  E_{\rho,\mu}(z)= \int_{\infty}^{1+\epsilon} K_{\rho,\mu}'(r,-\pi,z)dr+I_{C_\epsilon}+ \int_{1+\epsilon}^{\infty} K_{\rho,\mu}'(r,\delta_{2\rho},z)dr,
\end{equation*}
where the form of the functions $K_{\rho,\mu}'(r,\delta,z)$ is defined by (\ref{eq:K'_case2}). Now we will make use of the representation here (\ref{eq:lemm_MLF_int1_I_Ceps_case2}) for the integral $I_{C_\epsilon}$. As~a result, we obtain
\vspace{12pt}
\begin{multline*}
  E_{\rho,\mu}(z)= \int_{\infty}^{1+\epsilon} K_{\rho,\mu}'(r,-\pi,z)dr -\int_{1+\epsilon}^{0}K_{\rho,\mu}^\prime(r,\delta_{2\rho},z)dr
  -\int_{0}^{1-\varepsilon_1} K_{\rho,\mu}^\prime(r,-\pi,z)dr\\
  +\int_{-2\pi}^{-\pi}P_{\rho,\mu}^\prime(\varepsilon_1,\psi,-2,z)d\psi
  -\int_{1+\varepsilon_1}^{1+\epsilon}  K_{\rho,\mu}^\prime(r,-\pi,z)dr
  +\int_{1+\epsilon}^{\infty} K_{\rho,\mu}'(r,\delta_{2\rho},z)dr\\
  = \int_{0}^{\infty} K_{\rho,\mu}'(r,\delta_{2\rho},z)dr
  -\int_{0}^{1-\varepsilon_1} K_{\rho,\mu}^\prime(r,-\pi,z)dr
  +\int_{-2\pi}^{-\pi}P_{\rho,\mu}^\prime(\varepsilon_1,\psi,-2,z)d\psi
  -\int_{1+\varepsilon_1}^{\infty}  K_{\rho,\mu}^\prime(r,-\pi,z)dr,
\end{multline*}
where $\Re\mu<1+1/\rho$. It remains to consider how to change the condition (\ref{eq:z_cond_lemm}) in this case. Since, in~the case considered $\frac{1}{2}<\rho\leqslant1, \delta_{1\rho}=\pi, \pi/(2\rho)<\delta_{2\rho}\leqslant\pi$, then the condition (\ref{eq:z_cond_lemm}) takes the form $\frac{\pi}{2\rho}-\delta_{2\rho}+\pi<\arg z <-\frac{\pi}{2\rho}+2\pi$.  Thus, we have obtained the statement of the theorem  for the second~case.

\emph{Case 3.} Consider now the case $1/2<\rho\leqslant1$, $\pi/(2\rho)<\delta_{1\rho}\leqslant\pi$, $\delta_{2\rho}=\pi$.  In~this case the segment $\Gamma_2$ of the auxiliary contour $\Gamma$ (see~Figure~\ref{fig:loops_aux}) will run along the positive part of a real axis  and, thus, will run through a singular point $\zeta''=1$. It should be pointed out that in this case an argument of points of the segment is equal to $\arg\zeta=0$. Consequently, in~this case $\arg\zeta''=0$. That is why, as~in the previous case, we will change the contour $\Gamma$ in such a way that the contour bypasses this singular point leaving it outside the contour. The~contour formed $\Gamma''$ (Figure~\ref{fig:loop_Gamma''_aux_case3}) consists of the arc of the circle  $C_\epsilon=\{\zeta:\ -\delta_{1\rho}-\pi\leqslant\arg\zeta\leqslant0,\ |\zeta|=1+\epsilon\}$ with the center in the point $\zeta=0$  and radius  $1+\epsilon, \epsilon>0$, the~segment $AB=\{\zeta:\ \arg\zeta=0, 1+\epsilon\geqslant|\zeta|\geqslant 1+\varepsilon_1\}$, the~arc of the circle $C''_{1\varepsilon_1}$ with the center in the point $\zeta''=1,$ ($\arg\zeta''=0$) and radius $\varepsilon_1, (0<\varepsilon_1<1, \varepsilon_1<\epsilon)$, the~segment $CD=\{\zeta:\ \arg\zeta=0, 1-\varepsilon_1\geqslant|\zeta|\geqslant \varepsilon\}$, the~arc of the circle $C_\varepsilon=\{\zeta:\ 0\geqslant\arg\zeta\geqslant-\delta_{1\rho}-\pi,\ |\zeta|=\varepsilon\}$ with the center at the origin of coordinates and radius $\varepsilon>0$ and $\varepsilon<1-\varepsilon_1$ and the segment $EF=\{\zeta:\ \arg\zeta=-\delta_{1\rho}-\pi,\ \varepsilon\leqslant|\zeta|\leqslant 1+\epsilon\}$.

\begin{figure}[H]
  \centering
  \includegraphics[width=0.4\textwidth]{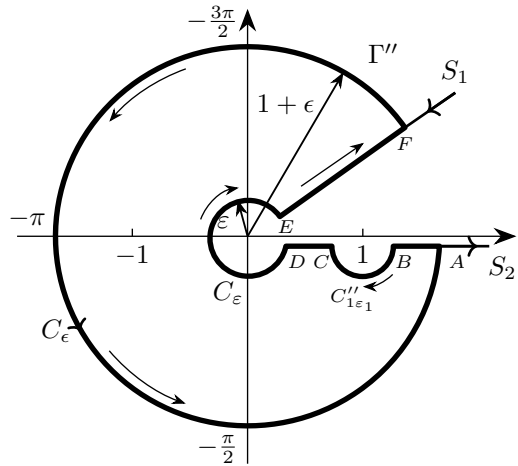}
  \caption{The auxiliary contour of integration $\Gamma''$.}\label{fig:loop_Gamma''_aux_case3}
\end{figure}

As one can see, only the arc of the circle $C''_{1\varepsilon_1}$ remains non-parametric. For~its parametrization we will fulfill the procedure as in the previous case. Considering the conformal mapping  $u=\zeta-1$ and representing the complex number $u$ in the form $u=\tau e^{i\psi}$ we get that the arc of the circle $C''_{1\varepsilon_1}$ can be represented in the form (\ref{eq:maps_C1_vareps}), where the parameter $\psi$ varies within the  limits from $-\pi$ to 0. Note that the value of the parameter $\psi=-\pi$ corresponds to the point $C$ and the value $\psi=0$ to the point $B$. It~should be pointed out that the center of a circle lies in the point $\zeta''=1$ with $\arg\zeta''=0$.   Consequently, in~the formulas (\ref{eq:maps_C1_vareps}) one should choose the principal branch $\arctan(x)$, i.e.,~to take $k=0$. As~a result, an~arc of a circle $C''_{1\varepsilon_1}$ can be represented in the form
\begin{equation}\label{eq:lemm_MLF_int1_C1vareps1_case3}
  C''_{1\varepsilon_1}=\left\{\zeta:\
  \begin{array}{c}
  \arg\zeta=\arctan\left(\frac{\tau\sin\psi}{\tau\cos\psi+1}\right) +k\pi,
  |\zeta|=\sqrt{\tau^2+2\tau\cos\psi+1},\\
  0\geqslant\psi\geqslant-\pi,\ \tau=\varepsilon_1,\ k=0.
  \end{array} \right\}
\end{equation}
In view of the foregoing, the~contour $\Gamma''$  is written in the form
\begin{equation}\label{eq:lemm_MLF_int1_Gamma''def}
  \Gamma''=\left\{\begin{array}{l}
                    C_\epsilon=\{\zeta:\ -\delta_{1\rho}-\pi\leqslant\arg\zeta\leqslant0,\ |\zeta|=1+\epsilon\},\\
                    AB=\{\zeta:\ \arg\zeta=0, 1+\epsilon\geqslant|\zeta|\geqslant 1+\varepsilon_1\},\\
                    C''_{1\varepsilon_1}=\left\{\zeta:\ \begin{array}{c}
                    \arg\zeta=\arctan\left(\frac{\tau\sin\psi}{\tau\cos\psi-1}+k\pi\right),\\
                    |\zeta|=\sqrt{\tau^2+2\tau\cos\psi+1},\\
                    0\geqslant\arg\zeta\geqslant-\pi, \tau=\varepsilon_1,k=0\end{array}\right\},\\
                    CD=\{\zeta:\ \arg\zeta=0, 1-\varepsilon_1\geqslant|\zeta|\geqslant \varepsilon\},\\
                    C_\varepsilon=\{\zeta:\ 0\geqslant\arg\zeta\geqslant-\delta_{1\rho}-\pi,\ |\zeta|=\varepsilon\},\\
                    EF=\{\zeta:\ \arg\zeta=-\delta_{1\rho}-\pi,\ \varepsilon\leqslant|\zeta|\leqslant 1+\epsilon\}.\\
                  \end{array}\right.
\end{equation}
The contour is traversed in a positive~direction.

We consider now an auxiliary integral
\begin{equation*}
  I''=\int_{\Gamma''}\frac{\phi_{\rho,\mu}(\zeta,z)}{\zeta-1}d\zeta.
\end{equation*}
In a similar way to how it was done in the previous case, we calculate this integral. Using the definition of the contour $\Gamma''$ we have
\begin{multline*}
  I''=\int_{C_\epsilon}\frac{\phi_{\rho,\mu}(\zeta,z)}{\zeta-1}d\zeta
  +\int_{AB}\frac{\phi_{\rho,\mu}(\zeta,z)}{\zeta-1}d\zeta +\int_{C''_{1\varepsilon_1}}\frac{\phi_{1,\mu}(\zeta,z)}{\zeta-1}d\zeta
  +\int_{CD}\frac{\phi_{\rho,\mu}(\zeta,z)}{\zeta-1}d\zeta \\
  +\int_{C_\varepsilon}\frac{\phi_{\rho,\mu}(\zeta,z)}{\zeta-1}d\zeta
  +\int_{EF}\frac{\phi_{\rho,\mu}(\zeta,z)}{\zeta-1}d\zeta
  = I_{C_\epsilon}+I''_{AB} + I''_{C''_{1\varepsilon_1}}+ I''_{CD} +I_{C_\varepsilon}+I''_{EF}.
\end{multline*}
It is clear from this expression that inside the region limited by the contour $\Gamma''$ the integrand of the integral  $I''$ is an analytical and continuous function of a variable $\zeta$. Consequently, according to the Cauchy theorem
\begin{equation*}
  I_{C_\epsilon}+I''_{AB} + I''_{C''_{1\varepsilon_1}}+ I''_{CD} +I_{C_\varepsilon}+I''_{EF}=0.
\end{equation*}
From this we obtain that
\begin{equation}\label{eq:lemm_MLF_int1_ICeps_summ_case3}
  I_{C_\epsilon}=-I''_{AB} - I''_{C''_{1\varepsilon_1}}- I''_{CD} -I_{C_\varepsilon}-I''_{EF}.
\end{equation}
We consider now each of the integrals in the right part~separately.

The integral $I_{C_\varepsilon}$ was considered by us earlier. Using (\ref{eq:IntIC_varepsTo0}), we obtain
\begin{equation}\label{eq:lemm_MLF_int1_I_Cvareps_case3}
  \lim_{\varepsilon\to0}I_{C_\varepsilon}=0,\quad \Re\mu<1+1/\rho.
\end{equation}

Now consider the integral $I''_{AB}$. Representing $\zeta$ in the form $\zeta=re^{i\varphi}$ we find
\begin{equation*}
  I''_{AB}=\int_{1+\epsilon}^{1+\varepsilon_1}\left.\frac{\phi_{\rho,\mu}(re^{i\varphi},z)}{re^{i\varphi}-1}e^{i\varphi} dr\right|_{\pi-\pi}= \int_{1+\epsilon}^{1+\varepsilon_1}\frac{\phi_{\rho,\mu}(re^{i(\pi-\pi)},z)}{re^{i(\pi-\pi)}-1}e^{i(\pi-\pi)} dr.
\end{equation*}
From this it is clear that this integral is similar to the integral (\ref{eq:I_AB_case2}).  That is why one can use the study results of this integral. Using (\ref{eq:K'_case2}), we obtain
\begin{equation}\label{eq:lemm_MLF_int1_I_AB_case3}
  I''_{AB}=\int_{1+\epsilon}^{1+\varepsilon_1} K'_{\rho,\mu}(r,\pi,z)dr.
\end{equation}

In the same way, using (\ref{eq:lemm_MLF_int1_Gamma''def}), (\ref{eq:lemm_MLF_int1_K'_def})  and  (\ref{eq:K'_case2}) for the integrals $I''_{CD}$ and $I''_{EF}$  we find
\begin{equation}\label{eq:lemm_MLF_int1_I_CD_I_EF_case3}
  I''_{CD}=\int_{1-\varepsilon_1}^{\varepsilon}K'_{\rho,\mu}(r,\pi,z)dr,\quad
I''_{EF}=\int_{\varepsilon}^{1+\epsilon} K'_{\rho,\mu}(r,\pi,z)dr.
\end{equation}

It remains to consider the integral $I_{C''_{1\varepsilon_1}}$. This integral is taken along the arc of the  $C''_{1\varepsilon_1}$ with the center in the point $\zeta''=1$, where  $\arg\zeta''=0$. This means that  the contour is traversed from the point $B$ of the complex plane  $\zeta$ in which $\arg\zeta_B=0$  to the point $C$ in which $\arg\zeta_C=0$. In~other words, the~starting and ending points of this contour have the same value of the argument. As~it was shown above, the~arc of a circle $C''_{1\varepsilon_1}$ can be represented in the form (\ref{eq:lemm_MLF_int1_C1vareps1_case3}).  Traversing this arc in the direction from the point $B$ to the point  $C$ corresponds to a change of the parameter  $\psi$ from 0 to $-\pi$. Thus, the~points of the arc of the circle $C''_{1\varepsilon_1}$ can be represented in the form
\begin{equation}\label{eq:lemm_MLF_int1_zeta_case3_C1}
  \zeta=r(\varepsilon_1,\psi)e^{i\varphi(\varepsilon_1,\psi,0)},
\end{equation}
where $r(\tau,\psi)$ and $\varphi(\tau,\psi,k)$ have the form (\ref{eq:lemm_MLF_r_varphi_C1vareps1}).

Using now the representation (\ref{eq:lemm_MLF_int1_zeta_case3_C1}), in~the integral $I_{C''_{1\varepsilon_1}}$ we obtain
\begin{equation*}
  I''_{C''_{1\varepsilon_1}}=\int_{0}^{-\pi}\left.\frac{\phi_{\rho,\mu}\left(r(\tau,\psi)e^{i\varphi(\tau,\psi,0)},z\right)i\tau e^{i\psi}} {r(\tau,\psi) e^{i\varphi(\tau,\psi,0)}-1}d\psi\right|_{\tau=\varepsilon_1}
  = -i\varepsilon_1 \int_{-\pi}^{0} \frac{\phi_{\rho,\mu}\left(r(\varepsilon_1,\psi)e^{i\varphi(\varepsilon_1,\psi,0)},z\right)e^{i\psi}} {r(\tau,\psi) e^{i\varphi(\tau,\psi,0)}-1}d\psi.
\end{equation*}
Comparing this expression with (\ref{eq:I_C1vareps1_case2_tmp0}) it is clear that these two integrals are similar. Therefore, using (\ref{eq:lemm_MLF_int1_P'_case2}) and (\ref{eq:lemm_MLF_int1_P'_def_case}) for the integral $I''_{C''_{1\varepsilon_1}}$ we get
\begin{equation}\label{eq:lemm_MLF_int1_I_C1vareps1_case3}
  I''_{C''_{1\varepsilon_1}}=-\int_{-\pi}^{0}P'_{\rho,\mu}(\varepsilon_1,\psi,0,z)d\psi.
\end{equation}

We return to the expression (\ref{eq:lemm_MLF_int1_ICeps_summ_case3}) and let  $\varepsilon\to0$ in this expression. Note that the integrals  $I_{C_\epsilon}$, $I''_{AB}$, $I''_{C''_{1\varepsilon_1}}$ do not depend on $\varepsilon$ and, consequently, with~such a passage to the limit they will not change. As~a result we obtain
\begin{equation*}
  I_{C_\epsilon}=-I''_{AB} - I''_{C''_{1\varepsilon_1}}- \lim_{\varepsilon\to0}I''_{CD} - \lim_{\varepsilon\to0}I_{C_\varepsilon}-\lim_{\varepsilon\to0}I''_{EF}.
\end{equation*}
Using here the expressions (\ref{eq:lemm_MLF_int1_I_Cvareps_case3}), (\ref{eq:lemm_MLF_int1_I_AB_case3}), (\ref{eq:lemm_MLF_int1_I_CD_I_EF_case3}), (\ref{eq:lemm_MLF_int1_I_C1vareps1_case3}) we obtain
\begin{multline*}
  I_{C_\epsilon}=-\int_{1+\epsilon}^{1+\varepsilon_1} K'_{\rho,\mu}(r,\pi,z)dr
  + \int_{-\pi}^{0}P'_{\rho,\mu}(\varepsilon_1,\psi,0,z)d\psi\\
  - \lim_{\varepsilon\to0}\int_{1-\varepsilon_1}^{\varepsilon}K'_{\rho,\mu}(r,\pi,z)dr
  - \lim_{\varepsilon\to0}\int_{\varepsilon}^{1+\epsilon} K'_{\rho,\mu}(r,\pi,z)dr, \quad\Re\mu<1+1/\rho.
\end{multline*}
As a result, performing the passage to the limit  we have
\begin{multline}\label{eq:lemm_MLF_int1_I_Ceps_case3}
  I_{C_\epsilon}=-\int_{1+\epsilon}^{1+\varepsilon_1} K'_{\rho,\mu}(r,\pi,z)dr
  + \int_{-\pi}^{0}P'_{\rho,\mu}(\varepsilon_1,\psi,0,z)d\psi
  - \int_{1-\varepsilon_1}^{0}K'_{\rho,\mu}(r,\pi,z)dr\\
  - \int_{0}^{1+\epsilon} K'_{\rho,\mu}(r,\pi,z)dr, \quad\Re\mu<1+1/\rho.
\end{multline}

We return now to the Mittag-Leffler function.  Assuming that $\delta_{2\rho}=\pi$ the contour of integration $\gamma_\zeta$ in (\ref{eq:lemm_MLF_int1_tmp0}) takes the form
\vspace{12pt}
\begin{equation*}
  \gamma_\zeta=\left\{\begin{array}{l}
                  S_1=\{\zeta:\ \arg\zeta=-\delta_{1\rho}-\pi,\ |\zeta|\geqslant 1+\epsilon\}\\
                  C_\epsilon=\{\zeta:\ -\delta_{1\rho}-\pi\leqslant \arg\zeta \leqslant 0,\ |\zeta|=1+\epsilon\} \\
                  S_2=\{\zeta:\ \arg\zeta=0,\ |\zeta|\geqslant 1+\epsilon\},
                \end{array}\right.
\end{equation*}
and the condition (\ref{eq:z_cond_lemm}) will be written in the form $ \frac{\pi}{2\rho}<\arg z< -\frac{\pi}{2\rho}+ \delta_{1\rho}+\pi$.

By direct calculating (\ref{eq:lemm_MLF_int1_tmp0}) and representing the complex number $\zeta$ in the form $\zeta=re^{i\varphi}$ we get
\begin{multline*}
  E_{\rho,\mu}(z)=\int_{S_1}\frac{\phi_{\rho,\mu}(\zeta,z)}{\zeta-1}d\zeta
  +\int_{C_{\epsilon}} \frac{\phi_{\rho,\mu}(\zeta,z)}{\zeta-1}d\zeta
  +\int_{S_1}\frac{\phi_{\rho,\mu}(\zeta,z)}{\zeta-1}d\zeta\\
  = \int_{\infty}^{1+\epsilon}\left.\frac{\phi_{\rho,\mu}\left(re^{i\varphi},z\right)}{re^{i\varphi}-1} e^{i\varphi}dr\right|_{\varphi=-\delta_{1\rho}-\pi}
  + I_{C_\epsilon}
  +\int_{1+\epsilon}^{\infty}\left.\frac{\phi_{\rho,\mu}\left(re^{i\varphi},z\right)}{re^{i\varphi}-1}e^{i\varphi}dr \right|_{\varphi=\pi-\pi}\\
  =\int_{\infty}^{1+\epsilon}\frac{\phi_{\rho,\mu}\left(re^{i(-\delta_{1\rho}-\pi)},z\right)}{re^{i(-\delta_{1\rho}-\pi)}-1} e^{i(-\delta_{1\rho}-\pi)}dr
  + I_{C_\epsilon}
  +\int_{1+\epsilon}^{\infty}\frac{\phi_{\rho,\mu}\left(re^{i(\pi-\pi)},z\right)}{re^{i(\pi-\pi)}-1} e^{i(\pi-\pi)}dr
\end{multline*}
using the notation (\ref{eq:lemm_MLF_int1_K'_def}) this expression can be written in the form
\begin{equation*}
  E_{\rho,\mu}(z)= \int_{\infty}^{1+\epsilon} K_{\rho,\mu}'(r,-\delta_{1\rho},z)dr+I_{C_\epsilon}+ \int_{1+\epsilon}^{\infty} K_{\rho,\mu}'(r,\pi,z)dr,
\end{equation*}
where the form of functions $K_{\rho,\mu}'(r,\delta,z)$ is defined  as (\ref{eq:K'_case2}). We substitute the integral  $I_{C_\epsilon}$ in this expression by the expression (\ref{eq:lemm_MLF_int1_I_Ceps_case3}). As~a result, we obtain
\begin{multline*}
  E_{\rho,\mu}(z)= \int_{0}^{1-\varepsilon_1} K_{\rho,\mu}'(r,\pi,z)dr
  + \int_{-\pi}^{0}P_{\rho,\mu}^\prime(\varepsilon_1,\psi,0,z)d\psi\\
  +\int_{1+\varepsilon_1}^{\infty} K_{\rho,\mu}^\prime(r,\pi,z)dr
  -\int_{0}^{\infty}  K_{\rho,\mu}^\prime(r,-\delta_{1\rho},z)dr,\quad \Re\mu<1+1/\rho.
\end{multline*}
 Thus, we have obtained the statement of the theorem  for the third~case.

\emph{Case 4.} Consider now the case $1/2<\rho\leqslant1$, $\delta_{1\rho}=\delta_{2\rho}=\pi$. In~this case the segments $\Gamma_1=\{\zeta:\ \arg\zeta=-2\pi,\ \varepsilon\leqslant|\zeta|\leqslant 1+\epsilon\}$ and $\Gamma_2=\{\zeta:\ \arg\zeta=0,\ \varepsilon\leqslant|\zeta|\leqslant 1+\epsilon\}$ of the auxiliary contour $\Gamma$ (see~Figure~\ref{fig:loops_aux}) run along the positive part of a real axis. However, the~argument of these segments  differs by the angle $2\pi$. Thus, the~segment $\Gamma_1$ passes through a singular point $\zeta'=1$ $(|\zeta'|=1,\arg\zeta'=-2\pi)$, and~the segment  $\Gamma_2$ passes through a singular point $\zeta''=1$ ($|\zeta''|=1, \arg\zeta''=0$). Consequently, the~contour $\Gamma$ must be changed so as to bypass these two points leaving them outside the~contour.

We consider the auxiliary contour $\Gamma^{\prime\prime\prime}$ (Figure~\ref{fig:loop_Gamma'''_case4}) consisting of the arc of the circle  $C_\epsilon$ with the center at the origin of coordinates and radius $1+\epsilon$, $\epsilon>0$, the~segment $AB$, the~arc of the circle $C''_{1\varepsilon_1}$ with the center in the point $\zeta''=1\ (|\zeta''|=1, \arg\zeta''=0)$ and radius $0<\varepsilon_1<\min(1,\epsilon)$ defined by the expression (\ref{eq:lemm_MLF_int1_C1vareps1_case3}), the~segment $CD$, the~arc of the circle $C_\varepsilon$ with the center at the origin of coordinates and radius $0<\varepsilon<1-\varepsilon_1$, the~segment $EF$, the~arc of the circle $C'_{1\varepsilon_1}$ with the center in the point $\zeta'=1\ (|\zeta'|=1, \arg\zeta'=-2\pi)$ and radius $0<\varepsilon_1<\min(1,\epsilon)$ having the form (\ref{eq:lemm_MLF_int1_C1vareps1_case2}) and the segment $GH$. The~contour is  traversed in a positive direction. As~a result, the~contour $\Gamma'''$ is written in the~form
\vspace{12pt}
\begin{equation*}\Gamma'''=\left\{
\begin{array}{l}
C_\epsilon =\{\zeta:\ -2\pi\leqslant\arg\zeta\leqslant0,\ |\zeta|=1+\epsilon\},\\
AB=\{\zeta: \arg\zeta=0, 1+\epsilon\geqslant|\zeta|\geqslant 1+\varepsilon_1\},\\
C''_{1\varepsilon_1}=\left\{\zeta:\
    \begin{array}{c}
    \arg\zeta=\arctan\left(\frac{\tau\sin\psi}{\tau\cos\psi+1}\right) +k\pi,\\
    |\zeta|=\sqrt{\tau^2+2\tau\cos\psi+1},\\
    0\geqslant\psi\geqslant-\pi,\ \tau=\varepsilon_1,\ k=0
    \end{array}\right\},\\
CD=\{\zeta:\ \arg\zeta=0, 1-\varepsilon_1\geqslant|\zeta|\geqslant\varepsilon\},\\
C_\varepsilon=\{\zeta:\ 0\geqslant\arg\zeta\geqslant-2\pi,\ |\zeta|=\varepsilon\},\\
EF=\{\zeta:\ \arg\zeta=-2\pi,\ \varepsilon\leqslant |\zeta| \leqslant 1-\varepsilon_1 \},\\
C'_{1\varepsilon_1}=\left\{\zeta:\
    \begin{array}{c}
        \arg\zeta=\arctan\left(\frac{\tau\sin\psi}{\tau\cos\psi+1}\right) +k\pi,\\
        |\zeta|=\sqrt{\tau^2+2\tau\cos\psi+1},\\
        -\pi\geqslant\psi\geqslant-2\pi,\ \tau=\varepsilon_1,\ k=-2,
    \end{array}\right\},\\
GH=\{\zeta:\ \arg\zeta=-2\pi,\ 1+\varepsilon_1\leqslant|\zeta|\leqslant1+\epsilon\}.
\end{array}\right.
\end{equation*}

\begin{figure}[H]
  \centering
  \includegraphics[width=0.4\textwidth]{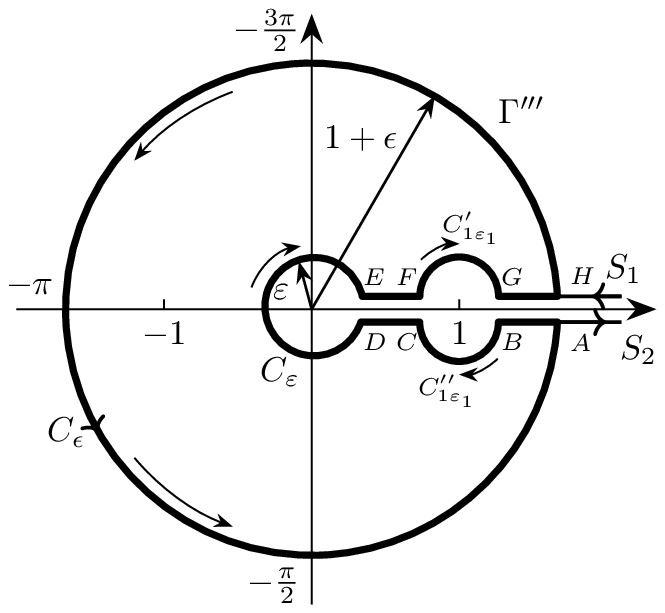}
  \caption{The auxiliary contour of integration $\Gamma'''$.}\label{fig:loop_Gamma'''_case4}
\end{figure}

Next, we consider an auxiliary integral
\begin{equation*}
  I'''=\int_{\Gamma'''}\frac{\phi_{\rho,\mu}(\zeta,z)}{\zeta-1}d\zeta,
\end{equation*}
where $\phi_{\rho,\mu}(\zeta,z)$ is defined by (\ref{eq:phiFun}). By~calculating this integral directly we obtain
\begin{multline*}
  I'''=\int_{\Gamma'''}\frac{\phi_{\rho,\mu}(\zeta,z)}{\zeta-1}d\zeta
  =\int_{C_\epsilon}\frac{\phi_{\rho,\mu}(\zeta,z)}{\zeta-1}d\zeta
  + \int_{AB}\frac{\phi_{\rho,\mu}(\zeta,z)}{\zeta-1}d\zeta
  + \int_{C''_{1\varepsilon_1}}\frac{\phi_{\rho,\mu}(\zeta,z)}{\zeta-1}d\zeta\\
  +\int_{CD}\frac{\phi_{\rho,\mu}(\zeta,z)}{\zeta-1}d\zeta
  +\int_{C_\varepsilon}\frac{\phi_{\rho,\mu}(\zeta,z)}{\zeta-1}d\zeta
  +\int_{EF}\frac{\phi_{\rho,\mu}(\zeta,z)}{\zeta-1}d\zeta
  +\int_{C'_{1\varepsilon_1}}\frac{\phi_{\rho,\mu}(\zeta,z)}{\zeta-1}d\zeta\\
  +\int_{GH}\frac{\phi_{\rho,\mu}(\zeta,z)}{\zeta-1}d\zeta
  = I_{C_\epsilon}+I'''_{AB}+I'''_{C''_{1\varepsilon_1}}+I'''_{CD}+I_{C_\varepsilon}+I'''_{EF}+I'''_{C'_{1\varepsilon_1}} +I'''_{GH}.
\end{multline*}

We immediately note that inside the region bounded by the contour $\Gamma'''$, the~function  $\phi_{\rho,\mu}(\zeta,z)/(\zeta-1)$ is an analytical function of the complex variable $\zeta$. Consequently, according to the Cauchy integral theorem,
\begin{equation*}
  I_{C_\epsilon}+I'''_{AB}+I'''_{C''_{1\varepsilon_1}}+I'''_{CD}+I_{C_\varepsilon}+I'''_{EF}+I'''_{C'_{1\varepsilon_1}} +I'''_{GH}=0.
\end{equation*}
From this it directly follows that
\begin{equation}\label{eq:lemm_MLF_int1_Ceps_summ_case4}
  I_{C_\epsilon}=-I'''_{AB}-I'''_{C''_{1\varepsilon_1}}-I'''_{CD}-I_{C_\varepsilon}-I'''_{EF}-I'''_{C'_{1\varepsilon_1}} -I'''_{GH}.
\end{equation}

Next, we calculate each integral on the right-hand side of this expression. However, these integrals were already calculated by us earlier. The~integral $I'''_{AB}$ corresponds to the integral $I''_{AB}$. Therefore, using (\ref{eq:lemm_MLF_int1_I_AB_case3}) we obtain
\begin{equation*}
  I'''_{AB}=\int_{1+\epsilon}^{1+\varepsilon_1} K_{\rho,\mu}^\prime(r,\pi,z)dr.
\end{equation*}
The integral $I'''_{CD}$ corresponds to the integral $I''_{CD}$. Using (\ref{eq:lemm_MLF_int1_I_CD_I_EF_case3}), we obtain
\begin{equation*}
  I'''_{CD}=\int_{1-\varepsilon_1}^{\varepsilon} K'_{\rho,\mu}(r,\pi,z)dr.
\end{equation*}
The integral $I'''_{C''_{1\varepsilon_1}}$ corresponds to the integral $I''_{C''_{1\varepsilon_1}}$. Using   (\ref{eq:lemm_MLF_int1_I_C1vareps1_case3})
\begin{equation*}
  I'''_{C''_{1\varepsilon_1}}=-\int_{-\pi}^{0}P'_{\rho,\mu}(\varepsilon_1,\psi,0,z)d\psi.
\end{equation*}
The integral $I_{C_\varepsilon}$ was considered by us when dealing with case 1. Using (\ref{eq:IntIC_varepsTo0}) for the case considered, we obtain
\begin{equation*}
  \lim_{\varepsilon\to0}I_{C_\varepsilon}=0,\quad\Re\mu<1+1/\rho.
\end{equation*}
The integral $I'''_{EF}$ corresponds to the integral $I'_{CD}$. Using (\ref{eq:lemma6_I_CD_case2}), we obtain
\begin{equation*}
  I'''_{EF}=\int_{\varepsilon}^{1-\varepsilon_1} K_{\rho,\mu}^\prime(r,-\pi,z)dr.
\end{equation*}
The integral $I'''_{GH}$ corresponds to the integral $I'_{EF}$. Therefore, using (\ref{eq:lemma6_I_EF_case2}) we obtain
\begin{equation*}
  I'''_{GH}=\int_{1+\varepsilon_1}^{1+\epsilon}  K_{\rho,\mu}^\prime(r,-\pi,z)dr.
\end{equation*}
The integral $I'''_{C'_{1\varepsilon_1}}$ corresponds to the integral $I'_{C'_{1\varepsilon_1}}$. Using (\ref{eq:lemma6_I_C1_vareps1_case2}), we obtain
\begin{equation*}
  I'''_{C'_{1\varepsilon_1}}=-\int_{-2\pi}^{-\pi}P_{\rho,\mu}^\prime(\varepsilon_1,\psi,-2,z)d\psi.
\end{equation*}

We get back to the expression (\ref{eq:lemm_MLF_int1_Ceps_summ_case4}) and let  $\varepsilon\to0$ in it. Note that the integrals $I'''_{AB}$, $I'''_{C''_{1\varepsilon_1}}$, $I'''_{C'_{1\varepsilon_1}}$, $I'''_{GH}$ do not depend on  $\varepsilon$. Therefore, with~such a passage to the limit, they do not change. Taking into account the above expressions for the integrals, the~sum  (\ref{eq:lemm_MLF_int1_Ceps_summ_case4}) takes the form
\begin{multline*}
  I_{C_\varepsilon}=-\int_{1+\epsilon}^{1+\varepsilon_1} K_{\rho,\mu}^\prime(r,\pi,z)dr +
  \int_{-\pi}^{0}P'_{\rho,\mu}(\varepsilon_1,\psi,0,z)d\psi
  - \lim_{\varepsilon\to0}\int_{1-\varepsilon_1}^{\varepsilon} K'_{\rho,\mu}(r,\pi,z)dr\\
  - \lim_{\varepsilon\to0} \int_{\varepsilon}^{1-\varepsilon_1} K_{\rho,\mu}^\prime(r,-\pi,z)dr
  + \int_{-2\pi}^{-\pi}P_{\rho,\mu}^\prime(\varepsilon_1,\psi,-2,z)d\psi
  - \int_{1+\varepsilon_1}^{1+\epsilon}  K_{\rho,\mu}^\prime(r,-\pi,z)dr,\quad \Re\mu<1+1/\rho.
\end{multline*}
Performing the passage to the limit in this expression $\varepsilon\to0$ we obtain
\begin{multline}\label{eq:lemm_MLF_int1_I_Ceps_case4}
  I_{C_\varepsilon}=-\int_{1+\epsilon}^{1+\varepsilon_1} K_{\rho,\mu}^\prime(r,\pi,z)dr +
  \int_{-\pi}^{0}P'_{\rho,\mu}(\varepsilon_1,\psi,0,z)d\psi
  - \int_{1-\varepsilon_1}^{0} K'_{\rho,\mu}(r,\pi,z)dr
  - \int_{0}^{1-\varepsilon_1} K_{\rho,\mu}^\prime(r,-\pi,z)dr\\
  + \int_{-2\pi}^{-\pi}P_{\rho,\mu}^\prime(\varepsilon_1,\psi,-2,z)d\psi
  - \int_{1+\varepsilon_1}^{1+\epsilon}  K_{\rho,\mu}^\prime(r,-\pi,z)dr,\quad\Re\mu<1+1/\rho.
\end{multline}

Now we consider the Mittag-Leffler (\ref{eq:lemm_MLF_int1_tmp0}). Taking into account that in the studied case $\delta_{1\rho}=\delta_{2\rho}=\pi$ we get
\begin{equation}\label{eq:MLF_rho1}
E_{\rho,\mu}(z)=\int_{\gamma_\zeta}\frac{\phi_{\rho,\mu}(\zeta,z)}{\zeta-1}d\zeta,
\end{equation}
where the contour $\gamma_\zeta$ takes the form
\begin{equation*}
  \gamma_\zeta=
  \left\{\begin{array}{l}
    S_1=\{\zeta:\ \arg \zeta=-2\pi,\ |\zeta|\geqslant 1+\epsilon\}, \\
    C_\varepsilon=\{\zeta:\ -2\pi \leqslant\arg \zeta \leqslant 0,\  |\zeta|=1+\epsilon,\\
    S_2=\{\zeta:\ \arg \zeta= 0,\ |\zeta|\geqslant 1+\epsilon\},
    \end{array}\right.
\end{equation*}
and the condition (\ref{eq:z_cond_lemm}) is written in the form $\tfrac{\pi}{2\rho}< \arg z< -\tfrac{\pi}{2\rho}+2\pi$.

Directly calculating this integral and representing the complex number  $\zeta$ in the form $\zeta=re^{i\varphi}$ we~have
\begin{multline*}
  E_{\rho,\mu}(z)=  \int_{S_1}\frac{\phi_{\rho,\mu}(\zeta,z)}{\zeta-1}d\zeta
  +\int_{C_\epsilon}\frac{\phi_{\rho,\mu}(\zeta,z)}{\zeta-1}d\zeta
  + \int_{S_2}\frac{\phi_{\rho,\mu}(\zeta,z)}{\zeta-1}d\zeta= \\
  = \int_{\infty}^{1+\epsilon}\left.\frac{\phi_{\rho,\mu}\left(re^{i\varphi},z\right)}{re^{i\varphi}-1} e^{i\varphi}dr\right|_{\varphi=-\pi-\pi}
  + I_{C_\epsilon}
  +\int_{1+\epsilon}^{\infty}\left.\frac{\phi_{\rho,\mu}\left(re^{i\varphi},z\right)}{re^{i\varphi}-1}e^{i\varphi}dr \right|_{\varphi=\pi-\pi}\\
  =\int_{\infty}^{1+\epsilon}\frac{\phi_{\rho,\mu}\left(re^{i(-\pi-\pi)},z\right)}{re^{i(-\pi-\pi)}-1} e^{i(-\pi-\pi)}dr
  + I_{C_\epsilon}
  +\int_{1+\epsilon}^{\infty}\frac{\phi_{\rho,\mu}\left(re^{i(\pi-\pi)},z\right)}{re^{i(\pi-\pi)}-1} e^{i(\pi-\pi)}dr.
\end{multline*}
Now using here (\ref{eq:lemm_MLF_int1_K'_def}) we obtain that the expression may be represented in the form
\begin{equation*}
  E_{\rho,\mu}(z)=\int_{\infty}^{1+\epsilon}K'_{\rho,\mu}(r,-\pi,z)dr+ I_{C_{\epsilon}} +\int_{1+\epsilon}^{\infty} K'_{\rho,\mu}(r,\pi,z)dr.
\end{equation*}

Now here we make use of the representation (\ref{eq:lemm_MLF_int1_I_Ceps_case4}) for the integral $I_{C_\epsilon}$. As~a result, we obtain
\begin{multline}\label{eq:lemm_MLF_int1_E1_case4_tmp0}
  E_{\rho,\mu}(z)=\int_{\infty}^{1+\epsilon}K'_{\rho,\mu}(r,-\pi,z)dr
  -\int_{1+\epsilon}^{1+\varepsilon_1} K_{\rho,\mu}^\prime(r,\pi,z)dr
  + \int_{-\pi}^{0}P'_{\rho,\mu}(\varepsilon_1,\psi,0,z)d\psi
  - \int_{1-\varepsilon_1}^{0} K'_{\rho,\mu}(r,\pi,z)dr\\
  - \int_{0}^{1-\varepsilon_1} K_{\rho,\mu}^\prime(r,-\pi,z)dr
  + \int_{-2\pi}^{-\pi}P_{\rho,\mu}^\prime(\varepsilon_1,\psi,-2,z)d\psi
  - \int_{1+\varepsilon_1}^{1+\epsilon}  K_{\rho,\mu}^\prime(r,-\pi,z)dr
  + \int_{1+\epsilon}^{\infty} K'_{\rho,\mu}(r,\pi,z)dr\\
  = \int_{0}^{1-\varepsilon_1}(K'_{\rho,\mu}(r,\pi,z)-K'_{\rho,\mu}(r,-\pi,z))dr
  +\int_{1+\varepsilon_1}^{\infty} (K'_{\rho,\mu}(r,\pi,z)-K'_{\rho,\mu}(r,-\pi,z))dr \\
  + \int_{-\pi}^{0}P'_{\rho,\mu}(\varepsilon_1,\psi,0,z)d\psi
  + \int_{-2\pi}^{-\pi}P_{\rho,\mu}^\prime(\varepsilon_1,\psi,-2,z)d\psi.
\end{multline}
In the expression derived there is the difference of $K'_{\rho,\mu}(r,\pi,z)-K'_{\rho,\mu}(r,-\pi,z)$. We transform this difference. For~this we use the representation (\ref{eq:K'_case2}) for the function $K'_{\rho,\mu}(r,\delta,z)$.
As a result, we obtain
\begin{multline*}
  K'_{\rho,\mu}(r,\delta,z)-K'_{\rho,\mu}(r,-\delta,z)\\
  =\frac{\rho}{2\pi i}\frac{\exp\left\{(zr)^\rho e^{-i\pi\rho}\cos(\rho\delta)\right\} (zre^{-i\pi})^{\rho(1-\mu)}}{r^2+2r\cos\delta+1}
   \left(e^{i\eta(r,\delta,z)}\left(r+e^{i\delta}\right)- e^{i\eta(r,-\delta,z)}\left(r+e^{-i\delta}\right)\right).
  \end{multline*}
From the definition for the function $\eta(r,\delta,z)$ (\ref{eq:eta_lemm_int0}) it follows that it is  an odd function relative to the parameter $\delta$
$$
\eta(r,-\delta,z)=-\eta(r,\delta,z),
$$

In the total, we have\begin{multline*}
K'_{\rho,\mu}(r,\delta,z)-K'_{\rho,\mu}(r,-\delta,z)\\
 = \frac{\rho}{2\pi i}\frac{\exp\left\{(zr)^\rho e^{-i\pi\rho}\cos(\rho\delta)\right\} (zre^{-i\pi})^{\rho(1-\mu)}}{r^2+2r\cos\delta+1}
  \left(e^{i\eta(r,\delta,z)}\left(r+e^{i\delta}\right)- e^{-i\eta(r,\delta,z)}\left(r+e^{-i\delta}\right)\right)\\
  \frac{\rho}{2\pi i}\frac{\exp\left\{(zr)^\rho e^{-i\pi\rho}\cos(\rho\delta)\right\} (zre^{-i\pi})^{\rho(1-\mu)}}{r^2+2r\cos\delta+1}
  \left(r\left(e^{i\eta(r,\delta,z)}-e^{-i\eta(r,\delta,z)}\right)
  + \left(e^{i(\eta(r,\delta,z)+\delta)}-e^{-i(\eta(r,\delta,z)+\delta)}\right)\right)\\
  = \frac{\rho}{2\pi i}\frac{\exp\left\{(zr)^\rho e^{-i\pi\rho}\cos(\rho\delta)\right\} (zre^{-i\pi})^{\rho(1-\mu)}}{r^2+2r\cos\delta+1}
  (r\sin(\eta(r,\delta,z))+\sin(\eta(r,\delta,z)+\delta))= K_{\rho,\mu}(r,\delta,z).
\end{multline*}
Here the function $K_{\rho,\mu}(r,\delta,z)$ was obtained by us earlier and has the form (\ref{eq:K_coroll_int0}). As~a result, the~expression (\ref{eq:lemm_MLF_int1_E1_case4_tmp0}) takes the form
\begin{equation*}
  E_{\rho,\mu}(z)= \int_{0}^{1-\varepsilon_1}K_{\rho,\mu}(r,\pi,z)dr
  +\int_{1+\varepsilon_1}^{\infty} K_{\rho,\mu}(r,\pi,z)dr
  + \int_{-\pi}^{0}P'_{\rho,\mu}(\varepsilon_1,\psi,0,z)d\psi
  + \int_{-2\pi}^{-\pi}P_{\rho,\mu}^\prime(\varepsilon_1,\psi,-2,z)d\psi.
\end{equation*}
The derived expression proves the theorem completely.
\end{proof}

The proved theorem shows that in the representation (\ref{eq:MLF_int}) at the values of parameters  $\rho, \delta_{1\rho}, \delta_{2\rho}$ satisfying the conditions (\ref{eq:deltaRhoCond_lem}) the integral over the arc of the circle  $C_\epsilon$ of the contour $\gamma_\zeta$ can be replaced with integration with respect to the segments $\Gamma_{1}$ and $\Gamma_2$ of the contour (\ref{eq:loop_Gamma_aux}). Such a transition makes it possible to replace the contour integral with one improper integral over the real variable of the complex-valued function, which simplifies the use and study of the Mittag-Leffler function. The~integral representation of the Mittag-Leffler function obtained in this theorem will be called the integral representation ``B''.

Consider in detail the differences between two derived representations. The~integral representation ``A''  (see (\ref{eq:MLF_int0})) of the Mittag-Leffler function consists of two summands. The~first summand corresponds to the sum of integrals along the half-lines  $S_1$ and $S_2$ and the second summand is the integral along the arc of a circle $C_\epsilon$. Such a representation is not convenient in analytical studies of the Mittag-Leffler function since one has to deal with two integrals. The~representation  ``B'' turns out to be more convenient and it consists of one improper integral  (see~(\ref{eq:MLF_int1})). However, this convenience is fraught with the appearance of constraints imposed on values of the parameters of the function $E_{\rho,\mu}(z)$. As~it was shown in  theorem~\ref{lemm:MLF_int1}, the~representation  ``B'' is true only at values  $\Re\mu<1+1/\rho$. If~$\Re\mu\geqslant1+1/\rho$, then it is necessary to use the representation ``A'' which is the main integral representation for the Mittag-Leffler function. It should also be noted that in the representation  ``B'' with the parameter values $1/2<\rho\leqslant1$ and values $\delta_{1\rho}=\pi$ or $\delta_{2\rho}=\pi$ the segments  $\Gamma_1$ or $\Gamma_2$ of the auxiliary contour $\Gamma$ (see~Figure~\ref{fig:loops_aux}) pass through the pole $\zeta=1$. This leads to the need to deform the contour of integration so as to bypass  the pole. As~a result, in~integral representations (\ref{eq:lemm_MLF_int1_case2}), (\ref{eq:lemm_MLF_int1_case3}), (\ref{eq:lemm_MLF_int1_case4}) there are summands describing the pole bypass along the arcs of the circle  $C'_{1\varepsilon_1}$ and $C''_{1\varepsilon_1}$. In~addition, one has to split the integrals along the half-lines $\Gamma_1+S_1$ and $\Gamma_2+S_2$ into parts to exclude the sections corresponding to the integrals along the arcs of the circle $C'_{1\varepsilon_1}$ and $C''_{1\varepsilon_1}$. All this leads to the appearance of additional terms in the corresponding integral representations. It should also be noted that at parameter values $1/2<\rho\leqslant1, \delta_{1\rho}=\delta_{2\rho}=\pi$ in the representation (\ref{eq:lemm_MLF_int1_case4}) the integral $\int_{0}^{1-\varepsilon_1}K_{\rho,\mu}(r,\pi,z)dr$ corresponds to the sum of integrals along the segments  $CD$ and $EF$ the contour of integration $\Gamma'''$ (see~Figure~\ref{fig:loop_Gamma'''_case4}) the integral $\int_{1+\varepsilon_1}^{\infty}K_{\rho,\mu}(r,\pi,z)dr$ corresponds to the sum of integrals along  $AB+S_2$ and $GH+S_1$. As~one can see, the specified sections of the integration contour go along the positive part of a real axis. As~it was shown in the lemma~\ref{lemm:MLF_SingPoints} at values of parameters $1/2<\rho<1$ or $\rho=1$ and  $\Im\mu\neq0$ or $\rho=1$, $\Im\mu=0$ and $\Re\mu$ is not an integer, the~point $\zeta=0$ is a branch point of the integrand (\ref{eq:lemm_MLF_int1_tmp0}). Consequently, the~indicated sections of the integration contour will go along different sides of the cut of the complex plane $\zeta$. As~a result, $\int_{0}^{1-\varepsilon_1}K_{\rho,\mu}(r,\pi,z)dr\neq0$ and $\int_{1+\varepsilon_1}^{\infty}K_{\rho,\mu}(r,\pi,z)dr\neq0$. In~case, if~$\rho=1$,  $\Im\mu=0$ and $\Re\mu$ is an integer,  lemma~\ref{lemm:MLF_SingPoints} shows that in this case  the point $\zeta=0$ is a regular point. Consequently, the~segments  $CD$ and $EF$, as~well as the half-lines $AB+S_2$ and $GH+S_1$ will go along one straight line of the complex plane, but~in different directions, and~arcs of the circles $C'_{1\varepsilon_1}$ and $C''_{1\varepsilon_1}$ will close. As~a result, we obtain $\int_{0}^{1-\varepsilon_1}K_{1,\mu}(r,\pi,z)dr=0$ and $\int_{1+\varepsilon_1}^{\infty}K_{1,\mu}(r,\pi,z)dr=0$, and~the sum of integrals $\int_{-\pi}^{0}P_{1,\mu}^\prime(\varepsilon_1,\psi,0,z)d\psi+
\int_{-2\pi}^{-\pi}P_{1,\mu}^\prime(\varepsilon_1,\psi,-2,z)d\psi$ corresponds to the integral over the closed contour and will be equal to the residue in the point $\zeta=1$. Thus, we come to the condition of Corollary~\ref{coroll:MLF_case_rho=1}.

The formulas of the integral representation obtained in theorem~\ref{lemm:MLF_int1} are rather lengthy. They can be simplified and reduced to a simpler form.  The~representation ``B'' takes the simplest form in the case when $\delta_{1\rho}$ and $\delta_{2\rho}$ coincide, i.e.,~ $\delta_{1\rho}=\delta_{2\rho}=\delta_\rho$. We will formulate the result obtained in the form of a~corollary.

\begin{corollary}\label{coroll:MLF_int1_deltaRho}
For any real $\rho>1/2$, any complex $\mu$ and  $z$ satisfying the conditions $\Re\mu<1+\frac{1}{\rho}$ and
\begin{equation}\label{eq:argZ_cond_corol_int1_deltaRho}
\frac{\pi}{2\rho}-\delta_\rho+\pi<\arg z<-\frac{\pi}{2\rho}+\delta_\rho+\pi,
\end{equation}
the Mittag-Leffler function can be represented in the~form:
\begin{enumerate}
  \item at any real $\delta_\rho$ satisfying the conditions
  $ \frac{\pi}{2\rho}<\delta_\rho\leqslant\frac{\pi}{\rho}$ if $\rho>1$ and $\frac{\pi}{2\rho}<\delta_\rho<\pi$ if $1/2<\rho\leqslant1$
\begin{equation}\label{eq:MLF_int1_deltaRho}
    E_{\rho,\mu}(z)=\int_{0}^{\infty} K_{\rho,\mu}(r,\delta_\rho,z) dr,
  \end{equation}
  where $K_{\rho,\mu}(r,\delta_\rho,z)$ has the form (\ref{eq:K_coroll_int0}).
  \item at $\delta_\rho=\frac{\pi}{\rho}$ and $\rho>1$
\begin{equation} \label{eq:MLF_int1_piRho}
    E_{\rho,\mu}(z)=\int_{0}^{\infty} K_{\rho,\mu}(r,z) dr,
  \end{equation}
  where $K_{\rho,\mu}(r,z)$ has the form (\ref{eq:K_piRho_coroll_int0}).
\end{enumerate}
\end{corollary}

\begin{proof}
1) According to theorem~\ref{lemm:MLF_int1} the representation (\ref{eq:MLF_int1}) is true at values $\delta_{1\rho}$ and $\delta_{2\rho}$ satisfying the conditions (\ref{eq:deltaRhoCond_lem}). In~this corollary it is assumed that $\delta_{1\rho}=\delta_{2\rho}=\delta_\rho$. In~view of this assumption, the~conditions (\ref{eq:deltaRhoCond_lem}) take the form
$\frac{\pi}{2\rho}<\delta_\rho\leqslant\frac{\pi}{\rho}$ if  $\rho>1$ and $\frac{\pi}{2\rho}<\delta_\rho<\pi$ if $1/2<\rho\leqslant1$. The~representation (\ref{eq:MLF_int1}) will be written in the form
\begin{equation*}
  E_{\rho,\mu}(z)=\int_{0}^{\infty} K_{\rho,\mu}(r,-\delta_\rho,\delta_\rho,z)dr.
\end{equation*}

From this it is clear that it is necessary to consider how the kernel function $K_{\rho,\mu}(r,-\delta_\rho,\delta_\rho,z)$ is transformed at the parameter values specified. This issue was already considered when proving item~1 of Corollary~\ref{coroll:MLF_int0_deltaRho} (see~(\ref{eq:corol_Int0_K_tmp0})). It was shown there that
\begin{multline*}
  K_{\rho,\mu}(r,-\delta_\rho,\delta_\rho,z)=K_{\rho,\mu}(r,\delta_\rho,z)=\\
  =\frac{\rho}{\pi}(zre^{-i\pi})^{\rho(1-\mu)}\exp\left\{(zre^{-i\pi})^\rho \cos(\rho\delta_\rho)\right\} \frac{r\sin(\eta(r,\delta_\rho,z))+ \sin(\eta(r,\delta_\rho,z)+\delta_\rho)}{r^2+2r\cos\delta_\rho+1}.
\end{multline*}
The first part of the corollary is~proved.

2) Consider the case $\delta_\rho=\pi/\rho$. Note immediately that the value of an angle $\delta_\rho$ cannot exceed $\pi$. Consequently, this case is implemented only at values $\rho\geqslant1$. However, the~value $\rho=1$ should be excluded from consideration. In~fact, at~$\rho=1$ the segments $\Gamma_1$ and $\Gamma_2$ of the auxiliary contour $\Gamma$ (see Figure~\ref{fig:loops_aux}) pass through the pole $\zeta=1$. Consequently, it is necessary to deform the contour $\Gamma$ so as to bypass this pole.  However, this has already been done by us in item 4 of  theorem~\ref{lemm:MLF_int1}.   By~putting the value $\delta_\rho=\pi/\rho$ in the representation (\ref{eq:MLF_int1_deltaRho}) we obtain
\vspace{12pt}
\begin{equation*}
  E_{\rho,\mu}(z)=\int_{0}^{\infty} K_{\rho,\mu}(r,\pi/\rho,z)dr.
\end{equation*}

Using the definition of the function $\eta(r,\psi,z)$ (see~(\ref{eq:eta_lemm_int0})) we obtain $\eta(r,\pi/\rho,z)=(1-\mu)\pi$.
Now using this result when calculating the kernel function $K_{\rho,\mu}(r,\pi/\rho,z)$ and introducing the notation $K_{\rho,\mu}(r,\pi/\rho,z)\equiv K_{\rho,\mu}(r,z)$ we obtain
\begin{equation*}
  E_{\rho,\mu}(z)=\int_{0}^{\infty} K_{\rho,\mu}(r,z)dr,
\end{equation*}
where $K_{\rho,\mu}(r,z)$ has the form (\ref{eq:K_piRho_coroll_int0}).
\end{proof}

\section{Conclusions}

It has been shown in the paper that when passing from the integral representation formulated in  theorem~\ref{lemm:MLF_int} to integration over real variables, the~integral representation of the Mittag-Leffler function can be written in two forms: the representation ``A'' and ``B''. The~integral representation ``A'' was given in theorem~\ref{lemm:MLF_int0} and the representation ``B'' in theorem~\ref{lemm:MLF_int1}. Each of these representations has its own advantages and drawbacks. The~representation ``A'' is true for any complex $\mu$ and at values $1/2<\rho\leqslant1$ there is no need to bypass a pole in the point $\zeta=1$. This excludes the necessity to consider particular cases and greatly simplifies the presentation itself. These facts are the advantages of the representation ``A''. The~disadvantages of this representation include the fact that it consists of the sum of two integrals: improper and definite. This leads to certain difficulties when working with this representation, since one has to study the behavior of these two integrals. The~representation ``B'' is valid only for the parameter values $\mu$ satisfying the condition $\Re\mu<1+1/\rho$. Particular cases also arise at values $1/2<\rho\leqslant1$ in which one has to bypass a singular point $\zeta=1$. That is why, in~the problems in which there is a necessity to investigate these particular cases this form of representation is not very convenient. These facts are the disadvantage of the representation ``B''. In~all other cases the representation ``B'' consists of one improper integral and turns out to be more convenient in use than the representation ``A''. This is the advantage of the representation ``B''.

The forms of the representations ``A'' and ``B'' obtained in  theorems~\ref{lemm:MLF_int0}~and~\ref{lemm:MLF_int1} are given for the case of four parameters $\rho, \mu, \delta_{1\rho}, \delta_{2\rho}$. The~parameters $\rho$ and $\mu$ are referred to the Mittag-Leffler function directly, and~parameters $\delta_{1\rho}$ and $\delta_{2\rho}$  describe the integration contour. As~a result, in~the general case, the~integral representations ``A'' and ``B'' of the Mittag-Leffler function turn out to be four-parameter. For~definiteness, we call this four-parameter description parametrization 1. However, in~parametrization 1 though the representations ``A'' and ``B'' of the function $E_{\rho,\mu}(z)$ have a more general form but they are awkward enough. The~representations ``A'' and ``B'' take a simpler form if $\delta_{1\rho}=\delta_{2\rho}=\delta_{\rho}$. In~this case, the integral representations ``A'' and  ``B'' can be described in three parameters: $\rho, \mu, \delta_\rho$. We call this parametrization 2. An~even simpler view of the form ``A'' and ``B'' is adopted if $\delta_{1\rho}=\delta_{2\rho}=\pi/\rho$. In~this case integral representations of the Mittag-Leffler function can be described in two parameters $\rho$ and $\mu$. We call this case of parametrization---parametrization 3. Thus, the~integral representations of the Mittag-Leffler function in the form ``A'' and parameterizations 2 and 3 are given in Corollary~\ref{coroll:MLF_int0_deltaRho}, and~the representation ``B'' in parametrizations 2 and 3 in Corollary~\ref{coroll:MLF_int1_deltaRho}.

Taking into account of  the geometric meaning of the parameters $\delta_{1\rho}$ and $\delta_{2\rho}$ one can give the geometric interpretation of three introduced parametrizations of the representations ``A'' and ``B''. In~fact, the~parameters $\delta_{1\rho}$ and $\delta_{2\rho}$ describe an inclination angle of half-lines  $S_1$ and $S_2$ in the contour $\gamma_\zeta$ (see~Figure~\ref{fig:loop_gammaZeta}). Thus, in~case if the half-lines $S_1$ and $S_2$ independently lie in the range of angles   $\tfrac{\pi}{2\rho}<\delta_{1\rho}\leqslant\min\left(\pi,\tfrac{\pi}{\rho}\right),\quad \tfrac{\pi}{2\rho}<\delta_{2\rho}\leqslant\min\left(\pi,\tfrac{\pi}{\rho}\right)$,  then we have parametrization 1. If~these half-lines lie symmetrically in relation to the positive part of a real axis, then we obtain parametrizations 2 and 3. With~this, parametrization 3 corresponds to the angle of inclination $\delta_{1\rho}=\delta_{2\rho}=\pi/\rho$.

In conclusion it should be pointed out that the representations for the function $E_{\rho,\mu}(z)$, formulated in Theorems~\ref{lemm:MLF_int0}~and~\ref{lemm:MLF_int1} are valid for the vaues  $\arg z$ satisfying the condition $\frac{\pi}{2\rho}-\delta_{2\rho}+\pi<\arg z<-\frac{\pi}{2\rho}+\delta_{1\rho}+\pi$. This constraint appears as a result of the use of the proof of  theorem~\ref{lemm:MLF_int} of the integral representation for the gamma function obtained in the work~\cite{Saenko2020} (see Appendix in~\cite{Saenko2020}). The~presence of this constraint somewhat narrows the possibilities of using the obtained integral representations of the function $E_{\rho,\mu}(z)$. Nevertheless, it is possible to get rid of this constraint and expand the range of admissible values $\arg z$ to the entire complex plane. However, this requires additional studies that  are beyond the scope of this~paper.

\AcknowledgementSection
This work was supported by the Russian Foundation for Basic Research (projects No 19-44-730005, 20-07-00655, 18-44-730004) and the Ministry of Science and Higher Education of the Russian Federation (projects No. 0830-2020-0008, RFMEFI60719X0301).

The author thanks to M.~Yu.~Dudikov for translation the article into~English.

%\bibliographystyle{elsarticle-num}
%\bibliography{d:/bibliography/library}

\end{document}